\documentclass[10pt]{article}

\usepackage[dvipsnames]{xcolor}
\usepackage{amsmath,amsthm,amssymb}
\usepackage{mathrsfs}
\usepackage{mathtools}
\usepackage{thmtools} 
\usepackage{proof}    
\usepackage{enumitem}
\usepackage{algpseudocode,algorithm,algorithmicx}
\usepackage{booktabs} 
\usepackage{titlesec} 
\usepackage{wrapfig}
\usepackage{graphicx}
\usepackage{subcaption}

\usepackage{sidecap} 
\usepackage[font=small, labelfont=bf, labelsep=period]{caption}  
\newcommand*\Let[2]{\State #1 $\gets$ #2}

\usepackage{cancel}

\usepackage{tikz}
\tikzset{every picture/.style={line width=0.5pt}} 

\font\bigbold=cmbx12

\def\maketitlefull#1#2#3#4#5#6{
  \centerline {\bigbold #1}
  \smallskip
  \centerline {\bigbold #2}
  \medskip
  \centerline {\tensc #3}
  \medskip
  \centerline {\sl #4}
  \smallskip
  \centerline {\sl #5}
  \smallskip
  \centerline {\sl #6}
  \bigskip
  }
\def\maketitle#1{
  \centerline {\bigbold #1}
}

\titleformat{\section}{\bf}{\thesection.}{0.5em}{}
\titlespacing{\section}{0em}{1.5\bigskipamount}{\medskipamount}
\titleformat{\subsection}{\bf}{\thesubsection.}{0.5em}{}
\titlespacing{\subsection}{0em}{\bigskipamount}{\medskipamount}
\titleformat{\paragraph}[runin]{\bf}{}{0em}{}[.]
\titlespacing{\paragraph}{0em}{\medskipamount}{*0.75}

\font\tensc=cmcsc10

\def\case#1. {{\it Case #1}.\enspace}

\def\parlabel#1. {\medskip\noindent{\bf #1.}\enspace}

\def\sqr#1#2{{\vcenter{\vbox{\hrule height.#2pt
        \hbox{\vrule width.#2pt height#1pt \kern#1pt
          \vrule width.#2pt}
        \hrule height.#2pt}}}}

\def\slug{\quad\hbox{\kern1.5pt\vrule width2.5pt height6pt depth1.5pt\kern1.5pt}\medskip}

\setlist[enumerate]{itemsep=\smallskipamount,parsep=0pt,label={\rm \roman*)}}
\setlist[itemize]{itemsep=\smallskipamount,parsep=0pt}

\newtheorem{thm}{Theorem}
\newtheorem{cor}[thm]{Corollary}
\newtheorem{prop}[thm]{Proposition}
\newtheorem{lem}[thm]{Lemma}
\newtheorem*{lem*}{Lemma}

\newtheorem{conj*}{Conjecture}

\theoremstyle{definition}

\theoremstyle{remark}
\newtheorem{rem}{Remark}

\newcommand{\ind}[1]{\1_{[#1]}}

\newcommand{\E}{\mathbb{E}}

\renewcommand{\P}{\mathbb{P}}

\newcommand{\R}{\mathbb{R}}

\usepackage{bbm}
\newcommand{\1}{\mathbbm{1}}

\newcommand{\V}{\mathcal{V}}

\DeclarePairedDelimiterX{\inp}[2]{\langle}{\rangle}{#1, #2} 

\newcommand{\no}{\noindent}

\usepackage{calligra}
    \DeclareMathAlphabet{\mathcalligra}{T1}{calligra}{m}{n}
    \DeclareFontShape{T1}{calligra}{m}{n}{<->s*[2.2]callig15}{}

\renewcommand{\mathbb}{\mathbf}
\renewcommand{\mathbbm}{\mathbf}
\renewcommand{\V}{\mathbf{V}}

\newcommand{\Bin}{\mathrm{Bin}}

\begin{document}

\maketitlefull{Density estimation using cellular binary trees}{and an application to monotone densities}{Luc Devroye and Jad Hamdan}{School of Computer Science, McGill University}{and}{Department of Mathematics, Oxford University}
\medskip

\[
  \vbox{
    \noindent{\bf Abstract.}\enskip 
    Consider a density $f$ on $[0,1]$ that must be estimated from an i.i.d.\ sample $X_1,...,X_n$ drawn from $f$. 
    In this note, we study binary-tree-based histogram estimates
    that use recursive splitting of intervals.
    If the decision to split an interval is a (possibly randomized) function of the
    number of data points in the interval only, then we speak of an estimate of complexity one. We exhibit a universally consistent estimate of complexity one. If the decision to split is a function of the cardinalities of $k$ equal-length sub-intervals, then we speak of an estimate of complexity $k$. We propose an estimate of complexity two that can estimate any bounded monotone density on $[0,1]$ with optimal expected total variation error $O(n^{-1/3})$.
    \smallskip 
    
    
    \noindent{\bf MSC2020 subject classification.} Primary 62G07, 68Q87, 68W40; secondary 60C05, 60J85.
  }
\] 

\smallskip

\section{Introduction} \label{sec:intro}
\no 

We are concerned with the estimation of an unknown density $f$ on $[0,1]$ based on an i.i.d.\ sample $X_1,...,X_n$ drawn from it.
In particular, given a recursive partition of the space into intervals, one can simply use the partition-based histogram estimate: on a fixed interval $C$, $f$ is estimated by

    \[
        f_n(x)=\frac{N(C)}{n\lambda(C)}, \quad x\in C,
    \]
    where $C$ is the unique interval to which $x$ belongs, $\lambda(C)$ is the length (or Lebesgue measure) of $C$, and $N(C)$ is the number of $X_i$'s falling in $C$.
We impose two further design restrictions, one for convenience, and one motivated by distributed computation. 

The only partitions of $[0,1]$ allowed are dyadic. This lets us view each of the allowed partitions as a binary tree. The root represents $[0,1]$, the two children of the root represent $[0,1/2]$ and $[1/2,1]$, and at level $i$ in the tree, we have an equi-partition of $[0,1]$ into $2^i$ intervals of length $1/2^i$. The leaves in the tree correspond to intervals whose union is $[0,1]$. (We abuse the term partition, as we allow intervals to overlap in a border point.)  To estimate $f$, it suffices to construct this tree given the data.   The data mining community (Schmidberger, 2009 \cite{gabi}; Ram and Gray, 2011 \cite{ramgray}; Anderlini, 2016 \cite{anderlini}) refers to this general method as ``density estimation trees". We note that in the event where the tree is not full, the resulting partition's bins can have different lengths.

Our second restriction is motivated by distributed computation. At the root, we may decide to either split the root interval or make it a leaf.  If it is split, the data travel to their respective sub-trees. So, the data flow down the tree according to where they belong.  Every node in the tree must act similarly, so we will refer to nodes as cells. The decision to split has a complexity parameter $\kappa$.  When $\kappa = 1$, the decision can only depend upon $N$, the number of points that fall in the node's interval, $C$.  In particular, that decision can't depend upon the original $n$, and can be handled by an autonomous computer, or ``cell". For $\kappa > 1$, the decision only depends upon $N_1,\ldots, N_k$, the cardinalities of the $k$ equal-length sub-intervals of $C$.  This view was also eschewed by Biau and Devroye (2013) \cite{devroyebiau}.

Many things can go awry. For one thing, we could end up with an infinite tree without any leaves.  Or we could end up with just one node, the root.  In the former case, the estimate is not defined. In the latter case, one gets consistency only if $f$ itself is uniform $[0,1]$.  The study here is a small first step into this new  model of ``cellular" estimators (meaning those which satisfy our restrictions), noting that generalizations in many directions are possible. Most importantly, one could consider $2^d$-ary trees to partition $[0,1]^d$ for general $d$. Computing will likely become more distributed and miniaturized, making our model relevant.

We begin with the description of an estimate of complexity one that is universally consistent, i.e., for all densities $f$ on $[0,1]$,
$$
\int | f_n - f | \to 0
$$
in probability as $n \to \infty$.

Then we exhibit an estimate of complexity two that consistently estimates every bounded monotone decreasing density on $[0,1]$ at an optimal rate in $n$.
The decision to split is made when $N_1-N_2>\gamma \sqrt{N_1+N_2}$ for fixed universal design constant $\gamma > 0$. We will show that
\[
    \E\bigg\{\int_{0}^1|f_n(x)-f(x)|\,dx\bigg\} = O\bigg(\frac{B^{2/3}}{n^{1/3}}\bigg)
\]
where $B=f(0)$ is the value at the mode, and $f_n$ is the histogram estimate on the partition induced by the leaves. Also, when $B=\infty$, we still have consistency, i.e., $\E\{\int|f_n-f|\}\to 0$ as $n\to \infty$.

It is noteworthy that if $\mathcal{M}_B$ denotes the class of all monotone densities on $[0,1]$ bounded by $B$, then
\[
    \inf_{\mbox{all estimators $f_n$}}\,\,\sup_{f\in \mathcal{M}_B} \E\bigg\{\int_0^1|f_n(x)-f(x)|\,dx\bigg\} \geq \alpha\left(\frac{\log(1+B)}{n}\right)^{1/3}
\]
for a universal constant $\alpha>0$ (Birgé, 1983 \cite{birge}; see also Devroye and Györfi, 1985 \cite{l1view}). Our estimate achieves the minimax rate in $n$, albeit with a sub-optimal constant multiplicative factor. Several estimates achieve the minimax rate with the correct multiplicative factor, notably Grenander's histogram estimate (Grenander, 1956 \cite{Grenander}) (which uses a partition based on the smallest concave majorant of the distribution function) and Birg\'e's histogram estimate (Birg\'e, 1983 \cite{birge}) (which uses a partition of exponentially increasing widths). Standard histogram estimates of equal bin widths can at best achieve a rate proportional to $(B/n)^{1/3}$.

Our simple estimate does not use additional information that one may know about the density. Prior knowledge of the smoothness, for instance, is completely ignored. With explicit knowledge of smoothness in terms of H\"older coefficient $\beta\geq 1$, the minimax rate can be of the order of $n^{-\beta/(2\beta+1)}$, surpassing $n^{-1/3}$ if $\beta>1$. This rate can be achieved with the kernel density estimate (Wasserman, 2006 \cite{allofnps}). 

This paper aims to present the most straightforward and most general estimate using the design restriction outlined above, to showcase what can still be achieved despite this hurdle. For this reason, we do not give much importance to the smoothness issue discussed in the previous paragraph.  After stating the main results announced above, we discuss the size of the tree obtained for the monotone density estimate and address the computational complexity. We note the importance of Galton--Watson trees in the analysis: as every density is locally nearly uniform, the performance of our splitting rule on a uniform density $f$  explains the behaviour near the bottom of the tree. We will show for example that for uniform $f$, the binary tree is essentially an extinct Galton--Watson tree of constant expected size.
We end the paper with extensions of the design principle to estimate densities with special structures such as convex or concave densities, log-concave or log-convex densities, and unimodal densities spring to mind. 


\section{A universally consistent estimate of complexity one}

Any deterministic splitting rule of complexity $\kappa=1$ is doomed because one can't decide which number of points $N(C)$ in a given interval is large enough to stop splitting. However, randomization can be used in the design. Assume that we have a non-increasing function $\varphi: Z \to [0,1]$, and that our estimate of complexity one is:
$$
\hbox{\rm do not split $C$ when}~U \le \varphi (N),
$$
where $N = N(C)$ is the cardinality of the interval $C$ and $U$ is an independent uniform $[0,1]$ random variable.
In this case, we obtain universal consistency:

\begin{thm}\label{nulltheorem} Let $f$ be a probability density function on $[0,1]$. Then
    $$
    \int_0^1 | f_n - f| \to 0
    $$
in probability as $n \to \infty$,
provided that
$$
\lim_{n \to \infty} \varphi(n) = 0
$$
and
$$
\lim_{n \to \infty} \varphi(n) \log_2 (n) = \infty.
$$
\end{thm}

\medskip
Proved in the Appendix, Theorem \ref{nulltheorem}
raises new questions regarding the tree's size, which measures the total computation time. One would also need information on the expected number of steps required to find the partition to which a point $x$ belongs, as that would be proportional to the expected time to compute the density estimate at one point. In addition, the height of the tree would be of interest. Finally, the choice of $\varphi$ within the bounds outlined in Theorem \ref{nulltheorem} should be studied.

    
\section{An estimate of complexity two for monotone densities}
        
        When $f$ is monotone on $[0,1]$ and non-increasing, an interval $C$ with equal-length sub-intervals $C_1$ and $C_2$ of cardinalities $N_1$ and $N_2$ 
        can be split by the following rule of complexity two:
        $$
        \hbox{\rm split $C$ when}~
        N_1-N_2>\gamma \sqrt{N_1+N_2}
        $$
        for fixed universal design constant $\gamma > 1$, noting that one would expect $N_1\geq N_2$ by monotonicity of $f$.
        For the resulting estimate $f_n$ we obtain an explicit upper
        bound on the total variation error:
    
\begin{thm}\label{maintheorem} Let $f$ be a bounded non-increasing probability density function on $[0,1]$ with $B=f(0)$. Then
    \[\sup_{f\in \mathcal{M}_B}  \E\bigg\{\int_{0}^1|f_n(x)-f(x)|dx\bigg\} \leq \beta  \left(\frac{B^{2/3}}{n^{1/3}}\right)\]
     for a universal constant $\beta$ and large enough $n$.
\end{thm}

\begin{rem}
Note that if we partition $[0,1]$ into $k$ equal intervals, and let $f_n$ be the standard histogram estimate for these $k$ intervals, then 
\[
    \E\bigg\{\int_{0}^1 |f_n-f| \bigg\} \leq \E\bigg\{\int_0^1 |f_n-\E f_n| \bigg\} + \int_0^1 |\E f_n-f|,
\]  
where $\E\{f_n(x)\}=p(C)/\lambda(C)$, $x\in C$, and $p(C)=\int_C f$. Thus, $\E\{f_n(x)\}=kp(C)$.
A simple shifting argument shows that
\[
    \int_0^1 |\mathbb{E}f_n-f| \leq \frac{B}{k}.
\]
Also, 
\begin{align*}
  \E\bigg\{\int_0^1|f_n-\E f_n|\bigg\}&=\sum_{C}\E\bigg\{\bigg|\frac{N(C)}{n}-p(C)\bigg|\bigg\}\\
  &\leq \sqrt{\sum_{C}1\cdot\sum_{C} \E\bigg\{\bigg|\frac{N(C)}{n}-p(C)\bigg|^2\bigg\}}\\ &\leq \sqrt{\frac{k}{n}\sum_{C}p(C)} = \sqrt{\frac{k}{n}},
\end{align*}
and therefore, taking $k=\lceil (2B)^{2/3}n^{1/3}\rceil$ to optimize the sum, we obtain
\[
    \sup_{f\in \mathcal{M}_B} \E\bigg\{\int_0^1 |f_n-f|\bigg\} \leq \left(\frac{1}{2^{2/3}}+2^{1/3}+o(1) \right)\left(\frac{B}{n}\right)^{1/3}.
\]
Note that without knowledge of $B$, the histogram estimate does not have a better convergence rate than our handicapped estimate.

\end{rem}


\section{Monotone density estimate: algorithm and time complexity}

From an algorithmic standpoint, the splitting described above amounts to a branching process that constructs a binary tree. For any $x\in \R$ and sorted list of numbers $L$, let $i(L, x)$ denote the index of $x$ if it were inserted into $L$. Given a sorted list of size $n$ whose elements are an i.i.d. sample $X_1,..., X_n$ drawn from an unknown density $f$ on $[a,b]\subseteq \mathbb{R}$, the following recursive algorithm constructs a partition tree of $[a,b]$ according to our splitting rule. 
\begin{algorithm}
  \caption{Interval partitioning using a binary tree}
  \begin{algorithmic}[1]\label{algo}
    \Statex
    \Function{BuildTree}{$r$, $[X_1,...,X_n], [a,b]$} \Comment{$r$ is a tree node}
      \Let{$L$}{$i([X_1,..., X_n], (a+b)/2)$} \Comment{$L$ is the number of data points in the left half of $[a,b]$} 
      \Let{$R$}{$n-L$} \Comment{$R$ is the number of data points on the right half of $[a,b]$}
        \If{$L-R > \gamma\sqrt{n}$} \Comment{$\gamma$ is a parameter in $(0, \infty)$}
             
             \State \Call{BuildTree}{$r$.left, $[X_1,...,X_{L}], [a, (a+b)/2]$}
             
             \State \Call{BuildTree}{$r$.right, $[X_{L+1},...,X_{n}], [(a+b)/2, b]$}
        \Else 
            \Let{$r$.value}{$[a,b]$}
        \EndIf
    \EndFunction
    
    \State Initialize new tree node $r$ 
    \State \Call{BuildTree}{$r, [X_1,...,X_n], [a,b]$}
    \State \Return{$r$}
  \end{algorithmic}
\end{algorithm}

While not strictly necessary, the assumption that $[X_1,..., X_n]$ is sorted allows us to decide whether or not to split in logarithmic time by using binary search to compute the number of points on the left and right halves of $[a,b]$. It also greatly simplifies the algorithm's pseudo-code to construct left and right sub-lists when we perform a recursive call.

As a corollary to the results shown later in the paper, we can derive the following sub-linear upper bound on the expected runtime of our algorithm.
\begin{cor}\label{cor2}
    This algorithm's expected runtime is $O(n^{1/3}\log_2(n))$ if the input data are sorted.
\end{cor}
\begin{proof}
    See appendix.
\end{proof}


\section{Monotone density estimate: Galton--Watson trees and the uniform case}

We recall the definition of a Galton--Watson tree (see, e.g., Athreya and Ney, 1972 \cite{athreya:ney}): the number of offspring of each node in the tree is random and distributed as $Z$, where $Z\geq 0$ has a fixed distribution. All realizations of $Z$ are independent. If $\mathbb{E}Z=m<1$, then the expected size of the tree is $1/(1-m)$. See, e.g., Lyons and Peres, 2016 \cite{peres}.

As previously discussed, our splitting procedure can be viewed as a (randomly generated) binary tree of intervals which we will henceforth denote by $T_n$. An elegant connection to the theory of branching processes can be established when our data are sampled from a uniform distribution. More specifically, one can show that the resulting tree would closely resemble a Galton--Watson tree whose nodes have two children with probability $p_2= \P\{\mathcal{N}(0,1)>\gamma\}:=\Phi(\gamma)$ and no children with probability $p_0=1-p_2$ (where $\mathcal{N}(0,1)$ is a standard normal and $\gamma$ is the parameter chosen in the algorithm, which is assumed to be $\geq 1$ in this section).

Let $C$ be an arbitrary sub-interval of $[0,1]$ and assume that $C$ contains $N$ data points. The number of points in the left and right halves of $C$, denoted by $N_1$ and $N_2$, respectively, are binomial random variables with parameters $N$ and $1/2$. Noting that $2N_1-N=N_1-N_2$, the probability of splitting the interval is
\[
    \P\{N_1-N_2>\gamma \sqrt{N}\} = \P\bigg\{\frac{N_1-N/2}{\sqrt{N/4}}> {\gamma}\bigg\},
\]
which by the Berry-Esseen theorem (Berry, 1941 \cite{berry}, see also Petrov, 1975 \cite{petrov}) is equal to $\Phi(\gamma)+\theta/\sqrt{N}=:p_2$ for some $|\theta|\leq 1$. 

Let $\epsilon > 0$ be arbitrary, and let $T'_n$ be the subtree of $T_n$ in which all nodes $C$ (we refer to $C$ as a node as well as an interval associated with that node) contain at least $N_\epsilon := \lceil 1/(\epsilon\cdot \Phi(\gamma))^2\rceil$ points. Then for these nodes, the probability $p_2$ of splitting is smaller than
\[
   (1+\theta\cdot\epsilon)\Phi(\gamma).
\]
 We infer that for $\epsilon$ small enough,
\[
    \E\{|T'_n|\} \leq \frac{1}{1-2(1+\epsilon)\Phi(\gamma)}.
\]

Furthermore, every leaf of $T'_n$ is either a leaf of $T_n$ or an internal node containing less than $N_\epsilon$ points. In the latter case, we can derive a uniform upper bound for the expected size of sub-trees that hang from such leaves as a function of $\epsilon$.
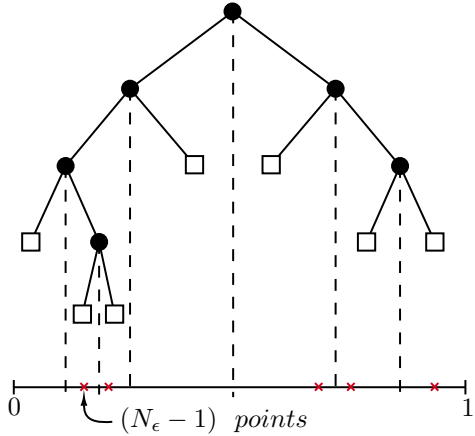
\begin{figure}
    \centering
    
\tikzset{every picture/.style={line width=0.75pt}} 

\begin{tikzpicture}[x=0.75pt,y=0.75pt,yscale=-.65,xscale=.65]

\draw    (150,323) -- (500.43,323) ;
\draw [shift={(500.43,323)}, rotate = 180] [color={rgb, 255:red, 0; green, 0; blue, 0 }  ][line width=0.75]    (0,5.59) -- (0,-5.59)   ;
\draw [shift={(150,323)}, rotate = 180] [color={rgb, 255:red, 0; green, 0; blue, 0 }  ][line width=0.75]    (0,5.59) -- (0,-5.59)   ;
\draw  [fill={rgb, 255:red, 0; green, 0; blue, 0 }  ,fill opacity=1 ] (234,91.21) .. controls (234,87.78) and (236.78,85) .. (240.21,85) .. controls (243.65,85) and (246.43,87.78) .. (246.43,91.21) .. controls (246.43,94.65) and (243.65,97.43) .. (240.21,97.43) .. controls (236.78,97.43) and (234,94.65) .. (234,91.21) -- cycle ;
\draw    (320.21,31.21) -- (240.21,91.21) ;
\draw  [fill={rgb, 255:red, 0; green, 0; blue, 0 }  ,fill opacity=1 ] (184,151.21) .. controls (184,147.78) and (186.78,145) .. (190.21,145) .. controls (193.65,145) and (196.43,147.78) .. (196.43,151.21) .. controls (196.43,154.65) and (193.65,157.43) .. (190.21,157.43) .. controls (186.78,157.43) and (184,154.65) .. (184,151.21) -- cycle ;
\draw    (240.21,91.21) -- (190.21,151.21) ;
\draw    (190.21,151.21) -- (163.21,210.21) ;
\draw  [fill={rgb, 255:red, 255; green, 255; blue, 255 }  ,fill opacity=1 ] (156.5,203.5) -- (169.93,203.5) -- (169.93,216.93) -- (156.5,216.93) -- cycle ;
\draw  [fill={rgb, 255:red, 0; green, 0; blue, 0 }  ,fill opacity=1 ] (210,210.21) .. controls (210,206.78) and (212.78,204) .. (216.21,204) .. controls (219.65,204) and (222.43,206.78) .. (222.43,210.21) .. controls (222.43,213.65) and (219.65,216.43) .. (216.21,216.43) .. controls (212.78,216.43) and (210,213.65) .. (210,210.21) -- cycle ;
\draw    (190.21,151.21) -- (216.21,210.21) ;
\draw    (240.21,91.21) -- (290.21,150.21) ;
\draw  [fill={rgb, 255:red, 255; green, 255; blue, 255 }  ,fill opacity=1 ] (283.5,143.5) -- (296.93,143.5) -- (296.93,156.93) -- (283.5,156.93) -- cycle ;
\draw  [dash pattern={on 4.5pt off 4.5pt}]  (190.21,151.21) -- (190.21,330.7) ;
\draw  [dash pattern={on 4.5pt off 4.5pt}]  (240.21,91.21) -- (240.21,330.7) ;
\draw  [dash pattern={on 4.5pt off 4.5pt}]  (216.21,210.21) -- (216.21,328.7) ;
\draw    (216.21,210.21) -- (203.21,266.21) ;
\draw    (216.21,210.21) -- (228.21,266.21) ;
\draw  [fill={rgb, 255:red, 255; green, 255; blue, 255 }  ,fill opacity=1 ] (196.5,259.5) -- (209.93,259.5) -- (209.93,272.93) -- (196.5,272.93) -- cycle ;
\draw  [fill={rgb, 255:red, 255; green, 255; blue, 255 }  ,fill opacity=1 ] (221.5,259.5) -- (234.93,259.5) -- (234.93,272.93) -- (221.5,272.93) -- cycle ;
\draw  [dash pattern={on 4.5pt off 4.5pt}]  (320.21,31.21) -- (320.21,330.7) ;
\draw  [fill={rgb, 255:red, 0; green, 0; blue, 0 }  ,fill opacity=1 ] (325.93,31.21) .. controls (325.93,27.78) and (323.15,25) .. (319.71,25) .. controls (316.28,25) and (313.5,27.78) .. (313.5,31.21) .. controls (313.5,34.65) and (316.28,37.43) .. (319.71,37.43) .. controls (323.15,37.43) and (325.93,34.65) .. (325.93,31.21) -- cycle ;
\draw  [fill={rgb, 255:red, 0; green, 0; blue, 0 }  ,fill opacity=1 ] (405.93,91.21) .. controls (405.93,87.78) and (403.15,85) .. (399.71,85) .. controls (396.28,85) and (393.5,87.78) .. (393.5,91.21) .. controls (393.5,94.65) and (396.28,97.43) .. (399.71,97.43) .. controls (403.15,97.43) and (405.93,94.65) .. (405.93,91.21) -- cycle ;
\draw    (319.71,31.21) -- (399.71,91.21) ;
\draw  [fill={rgb, 255:red, 0; green, 0; blue, 0 }  ,fill opacity=1 ] (455.93,151.21) .. controls (455.93,147.78) and (453.15,145) .. (449.71,145) .. controls (446.28,145) and (443.5,147.78) .. (443.5,151.21) .. controls (443.5,154.65) and (446.28,157.43) .. (449.71,157.43) .. controls (453.15,157.43) and (455.93,154.65) .. (455.93,151.21) -- cycle ;
\draw    (399.71,91.21) -- (449.71,151.21) ;
\draw    (449.71,151.21) -- (476.71,210.21) ;
\draw  [fill={rgb, 255:red, 255; green, 255; blue, 255 }  ,fill opacity=1 ] (483.43,203.5) -- (470,203.5) -- (470,216.93) -- (483.43,216.93) -- cycle ;
\draw    (449.71,151.21) -- (423.71,210.21) ;
\draw    (399.71,91.21) -- (349.71,150.21) ;
\draw  [fill={rgb, 255:red, 255; green, 255; blue, 255 }  ,fill opacity=1 ] (356.43,143.5) -- (343,143.5) -- (343,156.93) -- (356.43,156.93) -- cycle ;
\draw  [dash pattern={on 4.5pt off 4.5pt}]  (449.71,151.21) -- (449.71,330.7) ;
\draw  [dash pattern={on 4.5pt off 4.5pt}]  (399.71,91.21) -- (399.71,330.7) ;
\draw  [fill={rgb, 255:red, 255; green, 255; blue, 255 }  ,fill opacity=1 ] (430.43,203.5) -- (417,203.5) -- (417,216.93) -- (430.43,216.93) -- cycle ;
\draw [color={rgb, 255:red, 208; green, 2; blue, 27 }  ,draw opacity=1 ]   (201.7,319.7) -- (207.41,325.68) ;
\draw [color={rgb, 255:red, 208; green, 2; blue, 27 }  ,draw opacity=1 ]   (207.43,319.7) -- (201.7,325.7) ;

\draw [color={rgb, 255:red, 208; green, 2; blue, 27 }  ,draw opacity=1 ]   (220.7,319.7) -- (226.41,325.68) ;
\draw [color={rgb, 255:red, 208; green, 2; blue, 27 }  ,draw opacity=1 ]   (226.43,319.7) -- (220.7,325.7) ;

\draw [color={rgb, 255:red, 208; green, 2; blue, 27 }  ,draw opacity=1 ]   (383.7,319.7) -- (389.41,325.68) ;
\draw [color={rgb, 255:red, 208; green, 2; blue, 27 }  ,draw opacity=1 ]   (389.43,319.7) -- (383.7,325.7) ;

\draw [color={rgb, 255:red, 208; green, 2; blue, 27 }  ,draw opacity=1 ]   (408.7,319.7) -- (414.41,325.68) ;
\draw [color={rgb, 255:red, 208; green, 2; blue, 27 }  ,draw opacity=1 ]   (414.43,319.7) -- (408.7,325.7) ;

\draw [color={rgb, 255:red, 208; green, 2; blue, 27 }  ,draw opacity=1 ]   (473.7,319.7) -- (479.41,325.68) ;
\draw [color={rgb, 255:red, 208; green, 2; blue, 27 }  ,draw opacity=1 ]   (479.43,319.7) -- (473.7,325.7) ;

\draw    (226.43,349.84) .. controls (203.99,351.71) and (204.28,342.21) .. (204.41,332.8) ;
\draw [shift={(204.43,330.84)}, rotate = 90] [color={rgb, 255:red, 0; green, 0; blue, 0 }  ][line width=0.75]    (6.56,-1.97) .. controls (4.17,-0.84) and (1.99,-0.18) .. (0,0) .. controls (1.99,0.18) and (4.17,0.84) .. (6.56,1.97)   ;

\draw (143,329) node [anchor=north west][inner sep=0.75pt]    {$0$};
\draw (495,330.14) node [anchor=north west][inner sep=0.75pt]    {$1$};
\draw (230,336.55) node [anchor=north west][inner sep=0.75pt]    {$( N_{\epsilon } -1) \ \ points$};

\end{tikzpicture}
    \caption{Dyadic splitting until each interval contains at most one point.}
    \label{fig:0}
\end{figure}

Assuming $\gamma\geq1$, our splitting criterion is such that any interval with a single point is never split. By analyzing the expected minimum distance between any two points in an interval, we can determine an upper bound for the expected height (and in turn size) of a tree.

Consider $N_\epsilon$ uniformly distributed points on an interval (without loss of generality, $[0,1]$). Let $D$ be an integer random variable taking value $i$ when the minimum distance between two points of the interval lies in $(2^{-i-1}, 2^{-i}]$. 

We split $[0,1]$ dyadically until each interval has $0$ or $1$ point (as depicted in figure \ref{fig:0}). The expected number of internal nodes of this tree is 
\begin{align*}
    \sum_{\ell=0}^\infty 2^\ell\cdot  &\P\Bigg\{\bigg[0,\frac{1}{2^\ell}\bigg] \text{ contains at least 2 points}\Bigg\}\\
    &\leq
        \sum_{\ell=0}^\infty 2^\ell \cdot \binom{N_\epsilon-1}{2} \frac{1}{2^{2\ell}} \leq \frac{(N_\epsilon-1)^2}{2} \sum_{\ell=0}^\infty \frac{1}{2^\ell} = (N_\epsilon-1)^2
\end{align*}
where $\ell$ is the level number in the tree.

Thus, the expected size is $\leq 2(N_\epsilon-1)^2+1< 2N_\epsilon^2$ since the number of leaves equals the number of internal nodes plus one. 

We conclude that, under the assumption of uniformly distributed data, the expected tree size is finite and uniformly bounded over all values of $n$:
\begin{align*}
    \mathbb{E}\{|T_n|\}&\leq \inf_{\epsilon>0}\,\frac{1}{1-2(1+\epsilon)\Phi(\gamma)}\cdot 2N_\epsilon^2 \\
    &\leq \inf_{\epsilon>0}\,\frac{2}{1-2(1+\epsilon)\Phi(\gamma)} \cdot \left(\frac{1}{\big(\epsilon\cdot \Phi(\gamma)\big)^2}+1\right)^2\stackrel{\text{def}}{=}\varphi(\gamma) < \infty.
\end{align*}
Similar reasoning yields the following theorem.
\begin{thm}
    Let $f=1$ on $[0,1]$ and $f=0$ elsewhere. Then if $\Phi(\gamma)+\tfrac{1}{\gamma} <1/2$,
    \[
        \E\left\{\int_0^1 |f_n(x)-f(x)|dx\right\}=O\left(\frac{1}{\sqrt{n}}\right).
    \]
\end{thm}
\begin{rem}
    Any choice of $\gamma\geq 3$ ensures that this condition is satisfied.
\end{rem}
\begin{proof}
    Fix $x\in [0,1]$. Conditioning on the height $h$ of the leaf to which $x$ belongs (in the partition tree defining $f_n$) and using Cauchy-Schwarz, we get
 \begin{align*}
        \E\{|f_n(x)-f(x)|\} 
        &\leq \sum_{\ell \geq 0}\sqrt{\P\{h=\ell\}\,\E \{|f_n(x)-f(x)|^2\,|\,h=\ell\}} \\
        &=\frac{1}{\sqrt{n}}\sum_{\ell\geq 0}\sqrt{2^\ell\P\{h=\ell\}}.
 \end{align*}
    We know that for any node containing $N$ points, the probability of it splitting is bounded by $p_N=\Phi(\gamma)+1/\sqrt{N}$. However, the condition on $\gamma$ implies that $p_N<1/2$ uniformly. It follows that $\P\{h=\ell\}< (1/2)^\ell$, and the summation above is $O(1)$ as a geometric series. Applying Tonelli's theorem to $\E\{\int |f_n-f|\}$ thus yields the desired result.
\end{proof}


\section{The deterministic infinite tree}

\subsection{Notation, setup and main proposition}
Towards our goal of proving Theorem \ref{maintheorem}, we begin with the analysis of the infinite full binary tree depicted in figure \ref{fig:2} and denoted by $\mathcal{T_\infty}$. It is analogous to $T_n$ in that each node of $\mathcal{T_\infty}$ is associated with a sub-interval of $[0,1]$; more specifically, if $C_1,...,C_{2^\ell}$ are $\mathcal{T_\infty}$'s level $\ell$ nodes labelled left to right, then $C_i$ corresponds to the interval $[(i-1)/2^{\ell}, i/2^{\ell}]$ for any $1\leq i \leq 2^\ell$.  It helps to view the random tree $T_n$ as a subset of this infinite deterministic tree. 

As in the introduction, the left and right halves of a node $C\in\mathcal{T}_\infty$ are denoted by $C_1$ and $C_2$, respectively. It is said to be \textit{balanced} (and is uncoloured in figure \ref{fig:2}) if it satisfies
\begin{equation}\label{eq:idealsplit}
        p(C_1)-p(C_2)\leq \gamma \sqrt{\frac{{p(C)}}{n}},
\end{equation}
where $\gamma$ is the parameter previously defined for the algorithm $(\ref{algo}$) and, as above, $p(C)=\int_Cf$. The set of all such nodes is denoted by $\mathcal{B}$. All other (coloured) nodes are said to be \textit{unbalanced} and belong to $\mathcal{B}^c$, the complement of $\mathcal{B}$. Similarly, for any positive real number $\alpha$, we denote by $\mathcal{B}^{(\alpha)}$ the set of nodes satisfying
\[
    p(C_1)-p(C_2)\leq \alpha \gamma \sqrt{\frac{{p(C)}}{n}},
\]
noting that $\mathcal{B}=\mathcal{B}^{(1)}$. The integer $\ell^*$ is defined as
\[
    \ell^* = \min\bigg\{\ell \in \mathbb{Z}_{>0}: \frac{B}{2^{\ell+1}} \leq \gamma \cdot 2^{\ell/2}\sqrt{\frac{{B}}{n}}\bigg\}.
\]
We denote by $\mathcal{P}_j(\mathcal{T}_\infty)=\mathcal{P}_j$ the set of nodes of $\mathcal{T}_\infty$ with \textit{exactly} $j$ balanced ancestors. Lastly, for any node $C$, the average value of $f$ on $C$ is denoted by $f(C)$. 

Note that if we were to truncate $\mathcal{T}_\infty$ by deleting nodes that fall below those belonging to $\mathcal{B}\cap \mathcal{P}_0$, the resulting tree (with leaf set $\mathcal{B}\cap \mathcal{P}_0$) would be the tree generated by the algorithm (\ref{algo}) if every interval $C$ contained its \textit{expected} number of data points, $np(C)$ (in which case our splitting rule becomes the negation of ($\ref{eq:idealsplit}$)). If this were the case, the density estimate extracted from this tree would therefore, on each leaf $C$, be equal to $f(C)$.  We begin by showing that Theorem $\ref{maintheorem}$ holds for this function, as stated in the following proposition.

\begin{prop}\label{prop}
    Let $f$ be a bounded decreasing probability density function on $[0,1]$ with $B=f(0)<\infty$, and let $\mathcal{F}_n$ be the function that takes the value $f(C)$ on every $C\in\mathcal{P}_0\cap \mathcal{B}$. Then the $L_1$ distance between these two functions does not exceed $c_0\cdot(B^{2/3}/n^{1/3})$ for some constant $c_0\in \R_{>0}$ that does not depend on $B$ or $n$.
    
\end{prop}

\subsection{Preliminary results and lemmas}

The following three lemmas are needed to prove Proposition $\ref{prop}$.

    \begin{figure}
    \centering
    \begin{minipage}{0.5\textwidth}
      \centering

\tikzset{every picture/.style={line width=0.75pt}} 

\begin{tikzpicture}[x=0.75pt,y=0.75pt,yscale=-.8,xscale=.8]

\draw    (191,635.34) -- (176,663.34) ;
\draw    (191,635.34) -- (205.43,685.55) ;
\draw    (221,603.34) -- (250.43,729.55) ;
\draw    (281,574.34) -- (340.93,740.55) ;
\draw    (148.43,769.41) -- (432.43,769.41) ;
\draw [shift={(434.43,769.41)}, rotate = 180] [color={rgb, 255:red, 0; green, 0; blue, 0 }  ][line width=0.75]    (10.93,-3.29) .. controls (6.95,-1.4) and (3.31,-0.3) .. (0,0) .. controls (3.31,0.3) and (6.95,1.4) .. (10.93,3.29)   ;
\draw    (160,776.41) -- (160,571.27) ;
\draw [shift={(160,569.27)}, rotate = 90] [color={rgb, 255:red, 0; green, 0; blue, 0 }  ][line width=0.75]    (10.93,-3.29) .. controls (6.95,-1.4) and (3.31,-0.3) .. (0,0) .. controls (3.31,0.3) and (6.95,1.4) .. (10.93,3.29)   ;
\draw    (191,765.13) -- (191,774.41) ;
\draw    (221,766.13) -- (221,774.41) ;
\draw    (281,766.13) -- (281,774.41) ;
\draw    (340,765.27) -- (340,773.55) ;
\draw    (401,765.27) -- (401,773.55) ;
\draw  [dash pattern={on 4.5pt off 4.5pt}]  (191,635.34) -- (191,765.13) ;
\draw  [dash pattern={on 4.5pt off 4.5pt}]  (221,603.34) -- (221,771.41) ;
\draw  [dash pattern={on 4.5pt off 4.5pt}]  (281,574.34) -- (281,773.41) ;
\draw  [color={rgb, 255:red, 208; green, 2; blue, 27 }  ,draw opacity=1 ][fill={rgb, 255:red, 255; green, 255; blue, 255 }  ,fill opacity=1 ] (182.21,663.34) .. controls (182.21,659.91) and (179.43,657.13) .. (176,657.13) .. controls (172.57,657.13) and (169.79,659.91) .. (169.79,663.34) .. controls (169.79,666.77) and (172.57,669.55) .. (176,669.55) .. controls (179.43,669.55) and (182.21,666.77) .. (182.21,663.34) -- cycle ;
\draw  [color={rgb, 255:red, 208; green, 2; blue, 27 }  ,draw opacity=1 ][fill={rgb, 255:red, 255; green, 255; blue, 255 }  ,fill opacity=1 ] (211.64,685.55) .. controls (211.64,682.12) and (208.86,679.34) .. (205.43,679.34) .. controls (202,679.34) and (199.21,682.12) .. (199.21,685.55) .. controls (199.21,688.99) and (202,691.77) .. (205.43,691.77) .. controls (208.86,691.77) and (211.64,688.99) .. (211.64,685.55) -- cycle ;
\draw  [color={rgb, 255:red, 208; green, 2; blue, 27 }  ,draw opacity=1 ][fill={rgb, 255:red, 255; green, 255; blue, 255 }  ,fill opacity=1 ] (256.64,729.55) .. controls (256.64,726.12) and (253.86,723.34) .. (250.43,723.34) .. controls (247,723.34) and (244.21,726.12) .. (244.21,729.55) .. controls (244.21,732.99) and (247,735.77) .. (250.43,735.77) .. controls (253.86,735.77) and (256.64,732.99) .. (256.64,729.55) -- cycle ;
\draw  [color={rgb, 255:red, 208; green, 2; blue, 27 }  ,draw opacity=1 ][fill={rgb, 255:red, 255; green, 255; blue, 255 }  ,fill opacity=1 ] (347.14,740.55) .. controls (347.14,737.12) and (344.36,734.34) .. (340.93,734.34) .. controls (337.5,734.34) and (334.71,737.12) .. (334.71,740.55) .. controls (334.71,743.99) and (337.5,746.77) .. (340.93,746.77) .. controls (344.36,746.77) and (347.14,743.99) .. (347.14,740.55) -- cycle ;
\draw [color={rgb, 255:red, 208; green, 2; blue, 27 }  ,draw opacity=1 ]   (190.43,663.27) -- (190.43,685.55) ;
\draw [color={rgb, 255:red, 208; green, 2; blue, 27 }  ,draw opacity=1 ]   (221,685.77) -- (221,729.55) ;
\draw [color={rgb, 255:red, 208; green, 2; blue, 27 }  ,draw opacity=1 ]   (280.43,729.55) -- (281.43,740.55) ;
\draw    (281,574.34) -- (221,603.34) ;
\draw    (221,603.34) -- (191,635.34) ;
\draw [color={rgb, 255:red, 208; green, 2; blue, 27 }  ,draw opacity=1 ]   (350.43,600.41) -- (370.43,600.41) ;
\draw    (350.43,620.41) -- (370.43,620.41) ;
\draw  [fill={rgb, 255:red, 255; green, 255; blue, 255 }  ,fill opacity=1 ] (287.21,574.34) .. controls (287.21,570.91) and (284.43,568.13) .. (281,568.13) .. controls (277.57,568.13) and (274.79,570.91) .. (274.79,574.34) .. controls (274.79,577.77) and (277.57,580.55) .. (281,580.55) .. controls (284.43,580.55) and (287.21,577.77) .. (287.21,574.34) -- cycle ;
\draw  [fill={rgb, 255:red, 255; green, 255; blue, 255 }  ,fill opacity=1 ] (227.21,603.34) .. controls (227.21,599.91) and (224.43,597.13) .. (221,597.13) .. controls (217.57,597.13) and (214.79,599.91) .. (214.79,603.34) .. controls (214.79,606.77) and (217.57,609.55) .. (221,609.55) .. controls (224.43,609.55) and (227.21,606.77) .. (227.21,603.34) -- cycle ;
\draw  [fill={rgb, 255:red, 255; green, 255; blue, 255 }  ,fill opacity=1 ] (197.21,635.34) .. controls (197.21,631.91) and (194.43,629.13) .. (191,629.13) .. controls (187.57,629.13) and (184.79,631.91) .. (184.79,635.34) .. controls (184.79,638.77) and (187.57,641.55) .. (191,641.55) .. controls (194.43,641.55) and (197.21,638.77) .. (197.21,635.34) -- cycle ;
\draw [color={rgb, 255:red, 208; green, 2; blue, 27 }  ,draw opacity=1 ]   (160,663.27) -- (190.43,663.27) ;
\draw [color={rgb, 255:red, 208; green, 2; blue, 27 }  ,draw opacity=1 ]   (190.21,685.55) -- (220.64,685.55) ;
\draw [color={rgb, 255:red, 208; green, 2; blue, 27 }  ,draw opacity=1 ][fill={rgb, 255:red, 208; green, 2; blue, 27 }  ,fill opacity=1 ]   (220.43,729.55) -- (280.43,729.55) ;
\draw [color={rgb, 255:red, 208; green, 2; blue, 27 }  ,draw opacity=1 ]   (281.43,740.55) -- (412.43,740.55) ;
\draw    (160.43,653.55) .. controls (188.43,657.55) and (169.43,682.55) .. (204.43,685.55) .. controls (239.43,688.55) and (213.43,728.55) .. (255.43,733.55) .. controls (297.43,738.55) and (333.43,738.55) .. (413.43,746.55) ;

\draw (170,778.27) node [anchor=north west][inner sep=0.75pt]  [font=\footnotesize]  {$0.125$};
\draw (206,778.27) node [anchor=north west][inner sep=0.75pt]  [font=\footnotesize]  {$0.25$};
\draw (268,778.27) node [anchor=north west][inner sep=0.75pt]  [font=\footnotesize]  {$0.5$};
\draw (325,778.27) node [anchor=north west][inner sep=0.75pt]  [font=\footnotesize]  {$0.75$};
\draw (395,778.27) node [anchor=north west][inner sep=0.75pt]  [font=\footnotesize]  {$1$};
\draw (370,613.41) node [anchor=north west][inner sep=0.75pt]  [font=\scriptsize]  {$f( x)$};
\draw (371,592.41) node [anchor=north west][inner sep=0.75pt]  [font=\scriptsize,color={rgb, 255:red, 208; green, 2; blue, 27 }  ,opacity=1 ]  {$f_{n}( x)$};
\draw (154,778.41) node [anchor=north west][inner sep=0.75pt]  [font=\footnotesize]  {$0$};

\end{tikzpicture}
\captionsetup{justification=centering}  
      \captionof{figure}{An estimate $f_n$ using a finite,\newline random tree $T_n$.}
      \label{fig:1}
    \end{minipage}%
    \begin{minipage}{.5\textwidth}
      \centering

\tikzset{every picture/.style={line width=0.75pt}} 

\begin{tikzpicture}[x=0.75pt,y=0.75pt,yscale=-.8,xscale=.8]

\draw  [fill={rgb, 255:red, 0; green, 0; blue, 0 }  ,fill opacity=1 ] (206.93,1034.34) .. controls (206.93,1030.91) and (204.15,1028.13) .. (200.71,1028.13) .. controls (197.28,1028.13) and (194.5,1030.91) .. (194.5,1034.34) .. controls (194.5,1037.77) and (197.28,1040.55) .. (200.71,1040.55) .. controls (204.15,1040.55) and (206.93,1037.77) .. (206.93,1034.34) -- cycle ;
\draw  [fill={rgb, 255:red, 0; green, 0; blue, 0 }  ,fill opacity=1 ] (296.93,1034.34) .. controls (296.93,1030.91) and (294.15,1028.13) .. (290.71,1028.13) .. controls (287.28,1028.13) and (284.5,1030.91) .. (284.5,1034.34) .. controls (284.5,1037.77) and (287.28,1040.55) .. (290.71,1040.55) .. controls (294.15,1040.55) and (296.93,1037.77) .. (296.93,1034.34) -- cycle ;
\draw  [fill={rgb, 255:red, 0; green, 0; blue, 0 }  ,fill opacity=1 ] (221.93,985.34) .. controls (221.93,981.91) and (219.15,979.13) .. (215.71,979.13) .. controls (212.28,979.13) and (209.5,981.91) .. (209.5,985.34) .. controls (209.5,988.77) and (212.28,991.55) .. (215.71,991.55) .. controls (219.15,991.55) and (221.93,988.77) .. (221.93,985.34) -- cycle ;
\draw  [dash pattern={on 4.5pt off 4.5pt}]  (306.71,905.34) -- (306.71,1080.7) ;
\draw  [fill={rgb, 255:red, 0; green, 0; blue, 0 }  ,fill opacity=1 ] (327.93,1034.34) .. controls (327.93,1030.91) and (325.15,1028.13) .. (321.71,1028.13) .. controls (318.28,1028.13) and (315.5,1030.91) .. (315.5,1034.34) .. controls (315.5,1037.77) and (318.28,1040.55) .. (321.71,1040.55) .. controls (325.15,1040.55) and (327.93,1037.77) .. (327.93,1034.34) -- cycle ;
\draw  [fill={rgb, 255:red, 0; green, 0; blue, 0 }  ,fill opacity=1 ] (357.93,1034.34) .. controls (357.93,1030.91) and (355.15,1028.13) .. (351.71,1028.13) .. controls (348.28,1028.13) and (345.5,1030.91) .. (345.5,1034.34) .. controls (345.5,1037.77) and (348.28,1040.55) .. (351.71,1040.55) .. controls (355.15,1040.55) and (357.93,1037.77) .. (357.93,1034.34) -- cycle ;
\draw  [fill={rgb, 255:red, 0; green, 0; blue, 0 }  ,fill opacity=1 ] (417.93,1034.34) .. controls (417.93,1030.91) and (415.15,1028.13) .. (411.71,1028.13) .. controls (408.28,1028.13) and (405.5,1030.91) .. (405.5,1034.34) .. controls (405.5,1037.77) and (408.28,1040.55) .. (411.71,1040.55) .. controls (415.15,1040.55) and (417.93,1037.77) .. (417.93,1034.34) -- cycle ;
\draw  [dash pattern={on 4.5pt off 4.5pt}]  (245.71,950.49) -- (245.71,1080.7) ;
\draw  [dash pattern={on 4.5pt off 4.5pt}]  (215.71,995.64) -- (215.71,1080.7) ;
\draw  [dash pattern={on 4.5pt off 4.5pt}]  (366.71,950.49) -- (366.71,1080.7) ;
\draw  [dash pattern={on 4.5pt off 4.5pt}]  (336.71,995.64) -- (336.71,1080.7) ;
\draw  [dash pattern={on 4.5pt off 4.5pt}]  (275.71,995.64) -- (275.71,1080.7) ;
\draw  [dash pattern={on 4.5pt off 4.5pt}]  (396.71,995.64) -- (396.71,1079.57) ;
\draw    (245.71,945.34) -- (275.71,985.34) ;
\draw    (366.71,945.34) -- (396.71,985.34) ;
\draw    (306.71,905.34) -- (366.71,945.34) ;
\draw    (306.71,905.34) -- (245.71,945.34) ;
\draw    (245.71,945.34) -- (215.71,985.34) ;
\draw    (215.71,985.34) -- (200.71,1034.34) ;
\draw    (215.71,985.34) -- (230.71,1034.34) ;
\draw    (275.71,985.34) -- (290.71,1034.34) ;
\draw    (275.71,985.34) -- (260.71,1034.34) ;
\draw    (336.71,985.34) -- (351.71,1034.34) ;
\draw    (336.71,985.34) -- (321.71,1034.34) ;
\draw    (396.71,985.34) -- (411.71,1034.34) ;
\draw    (366.71,945.34) -- (336.71,985.34) ;
\draw    (396.71,985.34) -- (381.71,1034.34) ;
\draw    (200.71,1034.34) -- (194.43,1051.7) ;
\draw  [dash pattern={on 0.84pt off 2.51pt}]  (194.43,1051.7) -- (190.43,1061.84) ;

\draw    (200.71,1034.34) -- (206.43,1051.7) ;
\draw    (260.71,1034.34) -- (254.43,1051.7) ;
\draw    (260.71,1034.34) -- (266.43,1051.7) ;
\draw    (290.71,1034.34) -- (284.43,1051.7) ;
\draw    (290.71,1034.34) -- (296.43,1051.7) ;
\draw    (351.71,1034.34) -- (345.43,1051.7) ;
\draw    (351.71,1034.34) -- (357.43,1051.7) ;
\draw    (321.71,1034.34) -- (315.43,1051.7) ;
\draw    (321.71,1034.34) -- (327.43,1051.7) ;
\draw    (411.71,1034.34) -- (405.43,1051.7) ;
\draw    (411.71,1034.34) -- (417.43,1051.7) ;
\draw    (381.71,1034.34) -- (375.43,1051.7) ;
\draw    (381.71,1034.34) -- (387.43,1051.7) ;
\draw  [fill={rgb, 255:red, 255; green, 255; blue, 255 }  ,fill opacity=1 ] (266.93,1034.34) .. controls (266.93,1030.91) and (264.15,1028.13) .. (260.71,1028.13) .. controls (257.28,1028.13) and (254.5,1030.91) .. (254.5,1034.34) .. controls (254.5,1037.77) and (257.28,1040.55) .. (260.71,1040.55) .. controls (264.15,1040.55) and (266.93,1037.77) .. (266.93,1034.34) -- cycle ;
\draw  [fill={rgb, 255:red, 255; green, 255; blue, 255 }  ,fill opacity=1 ] (281.93,985.34) .. controls (281.93,981.91) and (279.15,979.13) .. (275.71,979.13) .. controls (272.28,979.13) and (269.5,981.91) .. (269.5,985.34) .. controls (269.5,988.77) and (272.28,991.55) .. (275.71,991.55) .. controls (279.15,991.55) and (281.93,988.77) .. (281.93,985.34) -- cycle ;
\draw  [fill={rgb, 255:red, 255; green, 255; blue, 255 }  ,fill opacity=1 ] (251.93,945.34) .. controls (251.93,941.91) and (249.15,939.13) .. (245.71,939.13) .. controls (242.28,939.13) and (239.5,941.91) .. (239.5,945.34) .. controls (239.5,948.77) and (242.28,951.55) .. (245.71,951.55) .. controls (249.15,951.55) and (251.93,948.77) .. (251.93,945.34) -- cycle ;
\draw  [fill={rgb, 255:red, 255; green, 255; blue, 255 }  ,fill opacity=1 ] (312.93,905.34) .. controls (312.93,901.91) and (310.15,899.13) .. (306.71,899.13) .. controls (303.28,899.13) and (300.5,901.91) .. (300.5,905.34) .. controls (300.5,908.77) and (303.28,911.55) .. (306.71,911.55) .. controls (310.15,911.55) and (312.93,908.77) .. (312.93,905.34) -- cycle ;
\draw  [fill={rgb, 255:red, 255; green, 255; blue, 255 }  ,fill opacity=1 ] (387.93,1034.34) .. controls (387.93,1030.91) and (385.15,1028.13) .. (381.71,1028.13) .. controls (378.28,1028.13) and (375.5,1030.91) .. (375.5,1034.34) .. controls (375.5,1037.77) and (378.28,1040.55) .. (381.71,1040.55) .. controls (385.15,1040.55) and (387.93,1037.77) .. (387.93,1034.34) -- cycle ;
\draw  [fill={rgb, 255:red, 255; green, 255; blue, 255 }  ,fill opacity=1 ] (342.93,985.34) .. controls (342.93,981.91) and (340.15,979.13) .. (336.71,979.13) .. controls (333.28,979.13) and (330.5,981.91) .. (330.5,985.34) .. controls (330.5,988.77) and (333.28,991.55) .. (336.71,991.55) .. controls (340.15,991.55) and (342.93,988.77) .. (342.93,985.34) -- cycle ;
\draw  [fill={rgb, 255:red, 255; green, 255; blue, 255 }  ,fill opacity=1 ] (402.93,985.34) .. controls (402.93,981.91) and (400.15,979.13) .. (396.71,979.13) .. controls (393.28,979.13) and (390.5,981.91) .. (390.5,985.34) .. controls (390.5,988.77) and (393.28,991.55) .. (396.71,991.55) .. controls (400.15,991.55) and (402.93,988.77) .. (402.93,985.34) -- cycle ;
\draw  [fill={rgb, 255:red, 255; green, 255; blue, 255 }  ,fill opacity=1 ] (372.93,945.34) .. controls (372.93,941.91) and (370.15,939.13) .. (366.71,939.13) .. controls (363.28,939.13) and (360.5,941.91) .. (360.5,945.34) .. controls (360.5,948.77) and (363.28,951.55) .. (366.71,951.55) .. controls (370.15,951.55) and (372.93,948.77) .. (372.93,945.34) -- cycle ;
\draw    (177,1080.41) -- (440.43,1080.41) ;
\draw [shift={(442.43,1080.41)}, rotate = 180] [color={rgb, 255:red, 0; green, 0; blue, 0 }  ][line width=0.75]    (10.93,-3.29) .. controls (6.95,-1.4) and (3.31,-0.3) .. (0,0) .. controls (3.31,0.3) and (6.95,1.4) .. (10.93,3.29)   ;
\draw    (187,1075.84) -- (187,1086.55) ;
\draw    (424,1074.84) -- (424,1085.55) ;
\draw  [dash pattern={on 0.84pt off 2.51pt}]  (206.43,1051.7) -- (210.43,1062.55) ;
\draw    (230.71,1034.34) -- (224.43,1051.7) ;
\draw    (230.71,1034.34) -- (236.43,1051.7) ;
\draw  [fill={rgb, 255:red, 255; green, 255; blue, 255 }  ,fill opacity=1 ] (236.93,1034.34) .. controls (236.93,1030.91) and (234.15,1028.13) .. (230.71,1028.13) .. controls (227.28,1028.13) and (224.5,1030.91) .. (224.5,1034.34) .. controls (224.5,1037.77) and (227.28,1040.55) .. (230.71,1040.55) .. controls (234.15,1040.55) and (236.93,1037.77) .. (236.93,1034.34) -- cycle ;
\draw  [dash pattern={on 0.84pt off 2.51pt}]  (236.43,1051.7) -- (240.43,1062.55) ;
\draw  [dash pattern={on 0.84pt off 2.51pt}]  (266.43,1051.7) -- (270.43,1062.55) ;
\draw  [dash pattern={on 0.84pt off 2.51pt}]  (296.43,1051.7) -- (300.43,1062.55) ;
\draw  [dash pattern={on 0.84pt off 2.51pt}]  (327.43,1051.7) -- (331.43,1062.55) ;
\draw  [dash pattern={on 0.84pt off 2.51pt}]  (357.43,1051.7) -- (361.43,1062.55) ;
\draw  [dash pattern={on 0.84pt off 2.51pt}]  (387.43,1051.7) -- (391.43,1062.55) ;
\draw  [dash pattern={on 0.84pt off 2.51pt}]  (417.43,1051.7) -- (421.43,1062.55) ;
\draw  [dash pattern={on 0.84pt off 2.51pt}]  (224.43,1051.7) -- (220.43,1060.84) ;
\draw  [dash pattern={on 0.84pt off 2.51pt}]  (254.43,1051.7) -- (250.43,1060.84) ;
\draw  [dash pattern={on 0.84pt off 2.51pt}]  (284.43,1051.7) -- (280.43,1060.84) ;
\draw  [dash pattern={on 0.84pt off 2.51pt}]  (315.43,1051.7) -- (311.43,1060.84) ;
\draw  [dash pattern={on 0.84pt off 2.51pt}]  (345.43,1051.7) -- (341.43,1060.84) ;
\draw  [dash pattern={on 0.84pt off 2.51pt}]  (375.43,1051.7) -- (371.43,1060.84) ;
\draw  [dash pattern={on 0.84pt off 2.51pt}]  (405.43,1051.7) -- (401.43,1060.84) ;

\draw (181,1088.41) node [anchor=north west][inner sep=0.75pt]  [font=\footnotesize]  {$0$};
\draw (418,1088.41) node [anchor=north west][inner sep=0.75pt]  [font=\footnotesize]  {$1$};

\end{tikzpicture}      \captionof{figure}{A depiction of \newline an infinite tree $\mathcal{T}_\infty$.}
      \label{fig:2}
    \end{minipage}
    \end{figure}

\begin{lem}\label{lem4}
    For any $C\in \mathcal{T}_\infty$,
    \[
        p(C_1)-p(C_2) \leq \int_C |f-f(C)|\leq 2\big(p(C_1)-p(C_2)\big).
    \]
\end{lem}
\begin{proof}
     Let $x_0 := \sup \{x\in C:f(x)\geq f(C)\}$. Without loss of generality, assume $C=[0,1]$ and $p(C)>0$. Our result is clear when $f$ is constant on $C$, so we assume otherwise. Assume first that $x_0< 1/2$, and define
    \[
        A:=\int_0^{x_0}\big(f-f(C)\big),\quad B_1:=\int_{x_0}^{1/2}\big(f(C)-f\big),\quad B_2:=\int_{1/2}^{1}\big(f(C)-f\big).
    \]
    Our assumption on $f$ guarantees that $A$, $B_1$ and $B_2$ are all positive.
    It is clear that $\int_c|f-f(C)|=A+B_1+B_2$ and $p(C_1)-p(C_2)=A+(B_2-B_1)$, which shows the leftmost inequality. Note that $B_2\geq B_1$, since otherwise we would have 
    \[
        p(C_1)-p(C_2)=A+(B_2-B_1)<A,
    \]
    which would only be possible $|x_0-1/2|\geq 1/2$, forcing $x_0=0$, $A=0$ and $f$ to be constant. Using the fact that $A=B_1+B_2$ (by definition of $x_0$), 
    \[
        2\big(p(C_1)-p(C_2)\big)=A+B_1+B_2+2(B_2-B_1)\geq \int_C|f-f(C)|.
    \]
    The case $x_0>1/2$ can be taken care of similarly.
\end{proof}

\begin{lem} \label{lem5} Let $\ell\in \mathbb{Z}^+$ be fixed, and let $\mathcal{A}_\ell$ be the set of nodes in $\mathcal{T}_\infty$ of depth $\ell$. Then
\[
    \sum_{C\in \mathcal{A}_\ell}\big(p(C_1)-p(C_2)\big)\leq \frac{B}{2^{\ell+1}}.
\]
\end{lem}
\begin{proof}
    Let $\{C_i\}_{i=1}^{2^\ell}$ be an enumeration of $\mathcal{A}_\ell$ from left to right (where the leftmost node has $0$ as one of its interval endpoints). We have
    \begin{align*}
        \sum_{C \in \mathcal{A}_\ell} \big(p(C_1)-p(C_2)\big) &= \sum_{i=1}^{2^\ell} \big(p(C_i')-p(C_i'')\big) \\
        &\leq p(C_1')-p(C_{2^\ell}'') \leq p(C_1') \leq \frac{B}{2^{\ell+1}}.
    \end{align*}
\end{proof}

\begin{lem} \label{lem6} Let $\ell\in \mathbb{Z}^+$ be fixed, and let $\mathcal{A}_\ell$ be the set of nodes in $\mathcal{T}_\infty$ of depth $\ell$. Then
\[
    \sum_{C\in \mathcal{A}_\ell}\sqrt{\frac{p(C)}{n}}\leq 2^{\ell/2} \sqrt{\frac{B}{n}}.
\]
\end{lem}
\begin{proof}
    By Jensen's inequality, $\sqrt{f(C)}\leq \int_C \sqrt{f}/\lambda(C)$ and
    \[
        \sqrt{p(C)}=\sqrt{\lambda(C)f(C)}\leq \frac{\int_C \sqrt{f}}{\sqrt{\lambda(C)}}.
    \]
    It follows that 
    \[
        \sum_{C\in \mathcal{A}_\ell}\sqrt{\frac{p(C)}{n}}\leq 2^{\ell/2} \frac{1}{\sqrt{n}}\int_C\sqrt{f} \leq 
        2^{\ell/2} \sqrt{\frac{B}{n}}.
    \]
\end{proof}

\subsection{Proof of proposition \ref{prop}}

Armed with these lemmas, we may now prove Proposition \ref{prop}.

\begin{proof} The $L_1$ distance between $f$ and $f_n$ on the whole of $[0,1]$ can be computed by summing the error over the leaf set $\mathcal{B}\cap \mathcal{P}_0$, and is thus equal to
\[
    \sum_{C\in \mathcal{B}\cap \mathcal{P}_0}\int_{C}|f-f(C)|.
\]
Using Lemma $\ref{lem4}$ and the definition $\mathcal{B}\cap \mathcal{P}_0$, we can upper bound this quantity and write
\begin{align}
    \sum_{C\in \mathcal{B}\cap \mathcal{P}_0}\int_{C}|f-f(C)|\leq 2\cdot\sum_{C\in \mathcal{B}\cap \mathcal{P}_0}\big(p(C_1)-p(C_2)\big) \leq 2\cdot\sum_{C\in\mathcal{B}\cap \mathcal{P}_0 } \gamma \sqrt{\frac{p(C)}{n}}.
\end{align}
By Lemmas \ref{lem5} and \ref{lem6},
\begin{equation}\label{crossingsum}
     \sum_{C\in \mathcal{B}\cap \mathcal{P}_0}\int_{C}|f-f(C)|\leq 2\cdot \sum_{\ell=0}^\infty \, \min\left(\frac{B}{2^{\ell+1}},\gamma \cdot 2^{\ell/2}\sqrt{\frac{B}{n}}\right).
\end{equation}
Recall that $\ell^*=\min\{\ell \in \mathbb{Z}_{>0} : {B}/{2^{\ell+1}} \leq \gamma \cdot 2^{\ell/2}\sqrt{{{B}/{n}}}\}$ and note that $\ell^*$ is within 1 of 
\[  
    \log_2\left\{ \bigg(\frac{Bn}{4}\bigg)^{1/3} \bigg(\frac{1}{\gamma}\bigg)^{2/3}  \right\},
\]
and that the summation in ($\ref{crossingsum}$) is bounded above by
\begin{align*}
    2 \left(\sum_{\ell=0}^{\ell^*-1}\gamma\sqrt{\frac{B}{n}}2^{\ell/2}+\sum_{\ell=\ell^*}^\infty\frac{B}{2^{\ell+1}}\right) &\leq 2\gamma \sqrt{\frac{B}{n}}2^{(\ell^*-1)/2}\left(\frac{1}{1-1/\sqrt{2}}\right)+\frac{2B}{2^{\ell^*}}\\
    &\leq \frac{\gamma^{2/3}B^{2/3}}{n^{1/3}}\left(\frac{2^{7/6}}{(\sqrt{2}-1)}+{2^{5/3}}\right).
\end{align*}
This non-asymptotic bound is uniform over all bounded monotone densities $f$. 

\end{proof}

\subsection{Additional results regarding the infinite tree}

We conclude this section by stating a few properties of $\mathcal{T}_\infty$ in the following lemmas, which are proved in the appendix. 
The first is a deterministic bound on the number of unbalanced nodes, both at a given level $\ell\geq \ell^*$ (equation \eqref{eq:unbalancedL}) and in general (equations \eqref{eq:unbalanced_general}, \eqref{eq:unbalanced_alpha}). The second is a bound on the number of nodes in $\mathcal{T}_\infty$ with exactly $j$ balanced ancestors for a given positive integer $j$.

\begin{lem}\label{lem7}
    Let $\mathcal{A}_\ell$ denote the set of nodes in $\mathcal{T}_\infty$ of depth $\ell$. If $\ell\geq \ell^*$,
    \begin{equation}\label{eq:unbalancedL}
                |\mathcal{A}_\ell \setminus \mathcal{B}| \leq \frac{2\sqrt{2}}{\gamma}\frac{\sqrt{Bn}}{{2^{\ell/2}}},
    \end{equation}
    and 
    \begin{equation}\label{eq:unbalanced_general}
                |\mathcal{B}^c|\leq \frac{5 B^{1/3}n^{1/3}}{\gamma^{2/3}}.
    \end{equation}
    Furthermore, 
    \begin{equation}\label{eq:unbalanced_alpha}
            \sup_{\alpha>0}  \alpha \big|\big(\mathcal{B}^{(\alpha)}\big)^c\big| \leq \frac{5B^{1/3}n^{1/3}}{\gamma^{2/3}}.
    \end{equation}
\end{lem}
\begin{lem}\label{lem8}
    Recall that $\mathcal{B}$ is the subset of balanced nodes and that $\mathcal{B}^c$ is the subset of unbalanced nodes of $\mathcal{T}_\infty$.
    For any $j\in \mathbb{Z}_{>0}$, define $\mathcal{P}_j=\mathcal{P}_j(\mathcal{T}_\infty)$ to be the set of nodes of $\mathcal{T}_\infty$ with exactly $j$ ancestors in $\mathcal{B}$, then
    \[
        |\mathcal{P}_j| \leq (|\mathcal{B}^c|+1)\cdot2^j.
    \]
\end{lem}


\section{Proof of Theorem \ref{maintheorem}}

Using the results above, we return to the proof of Theorem \ref{maintheorem}. The \textit{expected} $L_1$ distance between $f$ and $f_n$ (as defined previously) is computed by summing over $T_n$'s leaf set, denoted by $L$. By Scheffé's identity (see Devroye and Györfi, 1985 \cite{l1view}), we have
\begin{align} \label{eq:eq1}
    \E\left\{\int_0^1|f-f_n|\right\}= 2\cdot\E\left\{\int_0^1(f-f_n)_+\right\}
\end{align}
where $(x)_+:=\max(x,0)$. Now, (\ref{eq:eq1}) is bounded from above by
\[
    2\underbrace{\E\Bigg\{\sum_{C\in L}\int_C(f-f(C))_+\Bigg\}}_{\mbox{(I)}}+2\underbrace{\E\Bigg\{\sum_{C\in L}\int_{C}\Big(f(C)-\frac{N(C)/n}{\lambda(C)}\Big)_+\Bigg\}}_{\mbox{(II)}}.
\]
Here $N(C)$ is the number of data points in $C$.
We bound each of these terms separately. 

We view $T_n$ as a sub-tree of $\mathcal{T}_\infty$. This allows us to recycle most of the notation introduced above. For instance, leaves of $T_n$ with depth $\ell$ are the elements of $L\cap \mathcal{A}_\ell$, while leaves that are balanced are elements of $L\cap \mathcal{B}$. 
 
\subsection{Upper bound for (I)}

We begin with a few preliminary results.

\begin{lem}\label{lem9}
    Let $C$ be any non-leaf node of $T_n$ with depth $\ell$ and let $D\subseteq L$ be the set of leaves of the sub-tree rooted at $C$, then
    \[
        \sum_{C^*\in D}\int_{C^*}\big(f-f(C^*)\big)_+\leq \int_C\big(f-f(C)\big)_+.
    \]
\end{lem}
    \begin{figure}
        \centering

\tikzset{every picture/.style={line width=0.75pt}} 

\begin{tikzpicture}[x=0.75pt,y=0.75pt,yscale=-1,xscale=1]

\draw    (215.06,498.77) -- (443.77,498.77) ;
\draw [shift={(445.77,498.77)}, rotate = 180] [color={rgb, 255:red, 0; green, 0; blue, 0 }  ][line width=0.75]    (7.65,-2.3) .. controls (4.86,-0.97) and (2.31,-0.21) .. (0,0) .. controls (2.31,0.21) and (4.86,0.98) .. (7.65,2.3)   ;
\draw    (225.02,502.47) -- (225.02,354.85) ;
\draw [shift={(225.02,352.85)}, rotate = 90] [color={rgb, 255:red, 0; green, 0; blue, 0 }  ][line width=0.75]    (7.65,-2.3) .. controls (4.86,-0.97) and (2.31,-0.21) .. (0,0) .. controls (2.31,0.21) and (4.86,0.98) .. (7.65,2.3)   ;
\draw    (347.89,354.5) -- (347.89,502.47) ;
\draw    (414.3,354.5) -- (414.3,503.34) ;
\draw    (225.5,369.28) .. controls (295.23,371.01) and (316.26,392.63) .. (335.08,422.04) .. controls (353.89,451.44) and (373.82,467.88) .. (414.77,468.74) ;
\draw    (214.43,389.09) -- (347.25,389.09) ;
\draw    (347.89,460.71) -- (424.73,460.71) ;
\draw   (225.02,505.08) .. controls (225.02,509.75) and (227.35,512.08) .. (232.02,512.08) -- (276.39,512.08) .. controls (283.06,512.08) and (286.39,514.41) .. (286.39,519.08) .. controls (286.39,514.41) and (289.72,512.08) .. (296.39,512.08)(293.39,512.08) -- (340.25,512.08) .. controls (344.92,512.08) and (347.25,509.75) .. (347.25,505.08) ;
\draw   (347.89,505.08) .. controls (347.89,509.75) and (350.22,512.08) .. (354.89,512.08) -- (371.42,512.08) .. controls (378.09,512.08) and (381.42,514.41) .. (381.42,519.08) .. controls (381.42,514.41) and (384.75,512.08) .. (391.42,512.08)(388.42,512.08) -- (407.77,512.08) .. controls (412.44,512.08) and (414.77,509.75) .. (414.77,505.08) ;
\draw    (214.43,414.3) -- (414.77,413.47) ;
\draw [color={rgb, 255:red, 208; green, 2; blue, 27 }  ,draw opacity=1 ]   (225.5,369.28) -- (225.5,388.52) ;
\draw [color={rgb, 255:red, 208; green, 2; blue, 27 }  ,draw opacity=1 ]   (225.5,388.52) -- (304.09,388.52) ;
\draw [color={rgb, 255:red, 208; green, 2; blue, 27 }  ,draw opacity=1 ]   (225.5,369.28) .. controls (281.95,371.05) and (293.02,383.38) .. (304.09,388.52) ;
\draw [color={rgb, 255:red, 74; green, 144; blue, 226 }  ,draw opacity=1 ]   (375.56,461.18) -- (414.77,461.18) ;
\draw [color={rgb, 255:red, 74; green, 144; blue, 226 }  ,draw opacity=1 ]   (414.77,468.74) -- (414.77,461.18) ;
\draw [color={rgb, 255:red, 74; green, 144; blue, 226 }  ,draw opacity=1 ]   (375.56,461.18) .. controls (391.53,466.61) and (401.49,468.66) .. (414.77,468.74) ;
\draw [color={rgb, 255:red, 208; green, 2; blue, 27 }  ,draw opacity=1 ]   (256.49,360.77) .. controls (246.2,360.77) and (247.39,370.55) .. (247.61,379.35) ;
\draw [shift={(247.63,381.32)}, rotate = 270] [color={rgb, 255:red, 208; green, 2; blue, 27 }  ,draw opacity=1 ][line width=0.75]    (6.56,-1.97) .. controls (4.17,-0.84) and (1.99,-0.18) .. (0,0) .. controls (1.99,0.18) and (4.17,0.84) .. (6.56,1.97)   ;
\draw [color={rgb, 255:red, 74; green, 144; blue, 226 }  ,draw opacity=1 ]   (392.64,445.96) .. controls (400.87,445.96) and (404.32,454.03) .. (405.65,462.53) ;
\draw [shift={(405.92,464.46)}, rotate = 263.12] [color={rgb, 255:red, 74; green, 144; blue, 226 }  ,draw opacity=1 ][line width=0.75]    (6.56,-1.97) .. controls (4.17,-0.84) and (1.99,-0.18) .. (0,0) .. controls (1.99,0.18) and (4.17,0.84) .. (6.56,1.97)   ;
\draw   (413.67,349.83) .. controls (413.67,345.16) and (411.34,342.83) .. (406.67,342.83) -- (335.25,342.83) .. controls (328.58,342.83) and (325.25,340.5) .. (325.25,335.83) .. controls (325.25,340.5) and (321.92,342.83) .. (315.25,342.83)(318.25,342.83) -- (232.5,342.83) .. controls (227.83,342.83) and (225.5,345.16) .. (225.5,349.83) ;
\draw  [dash pattern={on 4.5pt off 4.5pt}]  (304.09,388.52) -- (304.09,497.98) ;
\draw  [dash pattern={on 4.5pt off 4.5pt}]  (375.56,461.18) -- (375.56,498.98) ;

\draw (275.95,523.53) node [anchor=north west][inner sep=0.75pt]  [font=\footnotesize]  {$C_{1}$};
\draw (372.25,523.28) node [anchor=north west][inner sep=0.75pt]  [font=\footnotesize]  {$C_{2}$};
\draw (173.87,378.28) node [anchor=north west][inner sep=0.75pt]  [font=\footnotesize]  {$f( C_{1})$};
\draw (180.19,402.92) node [anchor=north west][inner sep=0.75pt]  [font=\footnotesize]  {$f( C)$};
\draw (429.56,451.24) node [anchor=north west][inner sep=0.75pt]  [font=\footnotesize]  {$f( C_{2})$};
\draw (257.92,350.41) node [anchor=north west][inner sep=0.75pt]  [font=\scriptsize,color={rgb, 255:red, 208; green, 2; blue, 27 }  ,opacity=1 ]  {$A$};
\draw (379.68,437.75) node [anchor=north west][inner sep=0.75pt]  [font=\scriptsize,color={rgb, 255:red, 74; green, 144; blue, 226 }  ,opacity=1 ]  {$B$};
\draw (319.12,317.66) node [anchor=north west][inner sep=0.75pt]  [font=\footnotesize]  {$C$};
\draw (285.43,485.55) node [anchor=north west][inner sep=0.75pt]  [font=\footnotesize]  {$x_{1}$};
\draw (357.43,485.55) node [anchor=north west][inner sep=0.75pt]  [font=\footnotesize]  {$x_{2}$};

\end{tikzpicture}        \caption{Definitions used in the proof of Lemma \ref{lem9}.}
        \label{fig:3.5}
    \end{figure}
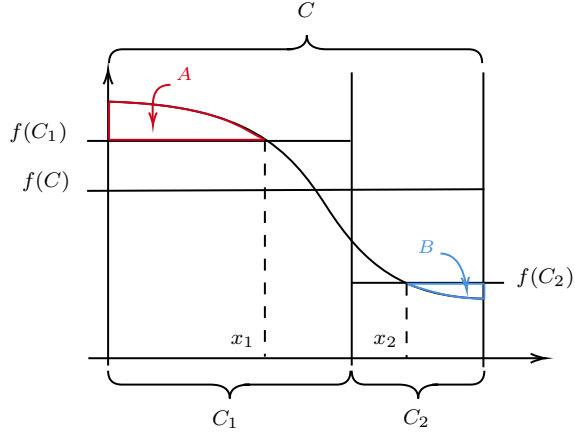
\begin{proof}
    See appendix.
\end{proof}

\begin{lem}\label{lem10}
    Let $C\in T_n \setminus \mathcal{B}^{(\sqrt{2})}$, and $\xi(C):=p(C_1)-p(C_2)-\gamma\sqrt{{2p(C)}/{n}}>0$. Then for such $C$, we have
    \[
        \P\{C\in L\} \leq \frac{2p(C)}{2p(C)+n\xi(C)^2}+\frac{4}{np(C)}.
    \]
\end{lem}
\begin{proof}
    See appendix.
\end{proof}

Using these lemmas, we prove the following proposition.
\begin{prop}\label{prop11}
    \[
        \sup_{f\in \mathcal{M}_B} \E\Bigg\{\sum_{C\in L}\int_C(f-f(C))_+\Bigg\} \leq \frac{B^{2/3}}{n^{1/3}} c_1(\gamma)+o(n^{-1/3})
    \]
    where 
\[
    c_1(\gamma):=\bigg(4\gamma^{2/3}+\frac{2\sqrt{2}(\gamma+\sqrt{\gamma^2+1})}{\gamma^{1/3}}\bigg)
\]
    is a strictly positive constant depending only upon $\gamma$.
\end{prop}

\begin{proof} The term we are trying to bound can be viewed as the expected $L_1$ distance between $f$ and the estimator obtained by taking the (random) partition of $[0,1]$ given by $T_n$, and estimating $f$ by its average value on each interval in the said partition. Informally, one notices that if the branching process that generated $T_n$ behaved ``as expected", this estimator would be more or less equal to $\mathcal{F}_n$ from Proposition \ref{prop}.

Deeper leaves in $T_n$ yield a finer partition of $[0,1]$. Taking intuition from the Riemann integral, one would guess that since we approximate $f$ by its average value on each interval of this partition, a finer partition would help us minimize $L_1$ distance. Conversely, we can use a coarser partition to upper bound said distance, as shown by Lemma \ref{lem9}.

Thus, we can use the partition given by $T_n$ truncated below level $\ell^*$ to derive our upper bound. By Lemmas \ref{lem4} and \ref{lem9}, we have
\begin{equation}\label{I.0}
    \mbox{(I)} \leq \E\left\{\sum_{\ell=0}^{\ell^*}\sum_{C\in L\cap \mathcal{A}_\ell}\int_C \big(f-f(C)\big)_+\right\}+2\cdot\E\Bigg\{\sum_{C\in A_{\ell^*}}\big(p(C_1)-p(C_2)\big)\Bigg\},
\end{equation}
and an application of Lemma \ref{lem5} yields 
\begin{equation} \label{eq:I.1}
    \E\Bigg\{\sum_{C\in A_{\ell^*}}\big(p(C_1)-p(C_2)\big)\Bigg\} \leq \frac{B}{2^{\ell^*+1}} \leq 2\frac{B^{2/3}\gamma^{2/3}}{n^{1/3}}.
\end{equation}
Next, we recall that $\mathcal{B}^{(\sqrt{2})}$ is the set of nodes of $\mathcal{T}_\infty$ satisfying
\[
    p(C_1)-p(C_2)\leq \gamma \sqrt{\frac{2p(C)}{n}},
\]
as defined earlier.
Any node $C$ belonging to the complement of $\mathcal{B}^{(\sqrt{2})}$ satisfies
\[
    p(C_1)-p(C_2)=\gamma\sqrt{\frac{2p(C)}{n}}+\xi(C)
\]
where $\xi(C):=p(C_1)-p(C_2)-\gamma\sqrt{{2p(C)}/{n}}$ is a strictly positive real number.

We use Lemma \ref{lem4} once more to bound the leftmost term in ($\ref{I.0}$), writing
\begin{align}\label{sum1}
    \E&\Bigg\{\sum_{\ell=0}^{\ell^*}\sum_{C\in L\cap \mathcal{A}_\ell}\int_C \big(f-f(C)\big)_+\Bigg\} \nonumber\\
    &= \E\Bigg\{\sum_{\ell=0}^{\ell^*}\sum_{C\in \mathcal{A}_\ell}\int_C \big(f-f(C)\big)_+\ind{C\in L}\Bigg\} \nonumber\\
    &\leq \sum_{\ell=0}^{\ell^*}\Bigg(\sum_{C\in \mathcal{B}^{(\sqrt{2})}\cap \mathcal{A}_\ell}  \gamma\sqrt{\frac{2p(C)}{n}} \nonumber \\
    &\quad\quad\quad+\sum_{C\in \mathcal{A}_\ell\setminus \mathcal{B}^{(\sqrt{2})}}\min\Big(\gamma\sqrt{\frac{2p(C)}{n}}+\xi(C), \,p(C_1)-p(C_2)\Big)\P\{{C\in L}\}\Bigg). 
\end{align}
Applying Lemma \ref{lem6}, we find that the first of the two inner summations in (\ref{sum1}) is bounded above by
\begin{equation}\label{I.1.1}
     \gamma \sqrt{\frac{B}{n}}2^{{(\ell+1)}/2}.
\end{equation}
To bound the second summation, we use Lemma \ref{lem10} as well as the fact that  $p(C_1)-p(C_2) \leq p(C)$ to write
\begin{align}\label{term1final}
    \sum_{C\in \mathcal{A}_\ell\setminus \mathcal{B}^{(\sqrt{2})}}&\min\left(\gamma\sqrt{\frac{2p(C)}{n}}+\xi(C), \,p(C_1)-p(C_2)\right)\P\{{C\in L}\} \nonumber \\
    &\leq \sum_{C\in \mathcal{A}_\ell\setminus \mathcal{B}^{(\sqrt{2})}}\Bigg(\Big(\gamma\sqrt{\frac{2p(C)}{n}}+\xi(C)\Big)\frac{1}{1+\xi(C)^2/(2p(C)/n)}+\frac{4}{n}\Bigg).
\end{align}
For any positive real numbers $a$ and $b$, the following identity holds:
\[
    \frac{a+b}{1+b^2}\leq \sqrt{a^2+1}.
\]
Using it with $a=\gamma$ and $b=\xi(C)/\sqrt{2p(C)/n}$ inside the summation in (\ref{term1final}), we have
\begin{equation}\label{eq:p6}
    \sum_{C\in \mathcal{A}_\ell\setminus \mathcal{B}^{(\sqrt{2})}}\bigg(\sqrt{\frac{2p(C)}{n}}\Big(\frac{a+b}{1+b^2}\Big)+\frac{4}{n}\bigg)\leq \sum_{C\in \mathcal{A}_\ell\setminus \mathcal{B}^{(\sqrt{2})}}\bigg(\sqrt{\frac{2p(C)}{n}} \sqrt{\gamma^2+1}+\frac{4}{n}\bigg).
\end{equation}
Since we are only bounding the quantity above for values of $\ell$ that are smaller than $\ell^*$, we have 
\[
    |\mathcal{A}_\ell \setminus \mathcal{B}^{(\sqrt{2})}|\leq |\mathcal{A}_\ell| \leq |A_{\ell^*}|\leq 2^{\ell^*}.
\]
By definition of $\ell^*$, this yields
\[
    |\mathcal{A}_\ell \setminus \mathcal{B}^{(\sqrt{2})}| \leq 2^{\ell^*} \leq \frac{2(Bn)^{1/3}}{\gamma^{2/3}}.
\]
Lemma \ref{lem6} implies that for any $\ell\leq \ell^*$,
\[
    \sum_{C\in \mathcal{A}_\ell\setminus \mathcal{B}^{(\sqrt{2})}}\bigg(\sqrt{\frac{2p(C)}{n}} \sqrt{\gamma^2+1}+\frac{4}{n}\bigg)\leq \left(\sqrt{\gamma^2+1} \right)\sqrt{\frac{B}{n}}2^{(\ell+1)/2}+\frac{8B^{1/3}}{(\gamma n)^{2/3}}.
\]

Invoking equations (\ref{eq:I.1}) and (\ref{I.1.1}), our bound on (I) in (\ref{I.0}) becomes
\begin{align*}
    \mbox{(I)} &\leq \frac{4B^{2/3}\gamma^{2/3}}{n^{1/3}}+\sum_{\ell=0}^{\ell^*}\left( \left(\gamma+\sqrt{\gamma^2+1}\right)\sqrt{\frac{B}{n}}2^{(\ell+1)/2}+\frac{8B^{1/3}}{(\gamma n)^{2/3}}\right) \\
    &\leq \frac{B^{2/3}}{n^{1/3}} c_1(\gamma)+\frac{8(\ell^*+1)B^{1/3}}{(\gamma n)^{2/3}}.
\end{align*}
So, $\ell^*/n^{2/3}=O(\log_2(n)/n^{2/3})$ uniformly over all monotone densities bounded by $B$. This completes the proof of Proposition \ref{prop11}. 
\end{proof}

\subsection{Upper bound for (II)}

Our upper bound for (II) is given in Proposition \ref{prop13} below. Its proof relies on the following preliminary result.

\begin{lem}\label{lem12} 
    Assume that $\gamma>1$. Then for any $0<\alpha<1-1/\gamma$ and $C\in \mathcal{B}^{(\alpha)}\cap \mathcal{P}_j$, we have
    \[
        \E\bigg\{\Big(p(C)-\frac{N(C)}{n}\Big)_+\ind{C\in L}\bigg\}\leq c_2(\gamma,\alpha)^{j/2}\sqrt{\frac{p(C)}{n}},
    \]
    where $c_2(\gamma,\alpha):=1/(1+(\gamma(1-\alpha))^{2})<1/2$.
\end{lem}
\begin{proof}
    See appendix.
\end{proof}

\begin{prop}\label{prop13}
Assume that $\gamma>1$. Then we have
\[
        \E\Bigg\{\sum_{C\in L}\int_{C}\Big(f(C)-\frac{N(C)/n}{\lambda(C)}\Big)_+\Bigg\} \leq  c_3(\gamma) \frac{B^{1/6}}{n^{1/3}}+o(n^{-1/3}),
\]
where
\[
    c_3(\gamma) = \inf_{0<\alpha<1-1/\gamma}\gamma^{-1/3}\sqrt{5}\cdot\bigg(\frac{1}{\sqrt{\alpha}}+\frac{1}{1-\sqrt{2c_2(\gamma,\alpha)}}\bigg).
\]
\end{prop}
\begin{proof} 
We start by writing
\begin{align}
    \E\Bigg\{\sum_{C\in L}\int_{C}\Big(f(C)-\frac{N(C)/n}{\lambda(C)}\Big)_+\Bigg\} = 
    & \, \E\bigg\{\sum_{C\in L} \Big(p(C)-\frac{N(C)}{n}\Big)_+\bigg\}.
\end{align}
 Let $0<\alpha<1-1/\gamma$ be arbitrary. We partition nodes $C\in L$ according to which $\mathcal{P}_j$ they belong to, as well as whether or not they belong to $\mathcal{B}^{(\alpha)}$, seeing that lemmas \ref{lem7} and \ref{lem8} provide upper bounds to the number of elements in these sets. We write
\begin{align}\label{finaltwo}
    \E&\bigg\{\sum_{C\in L} \Big(p(C)-\frac{N(C)}{n}\Big)_+\bigg\}\nonumber \\
    &\leq \sum_{j=0}^\infty \sum_{C\in \mathcal{B}^{(\alpha)} \cap \mathcal{P}_j} \E\bigg\{\Big(p(C)-\frac{N(C)}{n}\Big)_+\ind{C\in L}\bigg\}+\sum_{C\notin \mathcal{B}^{(\alpha)}}\sqrt{\frac{p(C)}{n}},
\end{align}
and then use the Cauchy-Schwarz inequality to obtain
\begin{align*}
    \sum_{C\notin \mathcal{B}^{(\alpha)}}\sqrt{\frac{p(C)}{n}} &\leq \sqrt{\bigg(\sum_{C\notin \mathcal{B}^{(\alpha)}}1\bigg)\bigg(\sum_{C\notin \mathcal{B}^{(\alpha)}}\frac{p(C)}{n}\bigg)}\\
    &\leq\frac{1}{\sqrt{n}} \sqrt{\big|\big(\mathcal{B}^{(\alpha)}\big)^c\big|\cdot \sum_{C\notin \mathcal{B}^{(\alpha)}}p(C)} \\
    &\leq \frac{1}{\sqrt{n}} \sqrt{\big|\big(\mathcal{B}^{(\alpha)}\big)^c\big|}.
\end{align*}
By Lemma \ref{lem7}, the latter is dominated by
\[
   \sqrt{\frac{5}{\alpha}}\cdot\frac{B^{1/6}}{\gamma^{1/3}n^{1/3}}.
\]
Further, we can use Lemma \ref{lem12} to write
\begin{align*}
    \sum_{j=0}^\infty \sum_{C\in \mathcal{B}^{(\alpha)} \cap \mathcal{P}_j} &\E\bigg\{\Big(p(C)-\frac{N(C)}{n}\Big)_+\ind{C\in L}\bigg\}&\\
    &\leq 
   \sum_{j=0}^\infty  \sum_{C\in \mathcal{B}^{(\alpha)} \cap \mathcal{P}_j} c_2(\gamma,\alpha)^{j/2}\sqrt{\frac{p(C)}{n}} &\\
    &\leq \frac{1}{\sqrt{n}} \sum_{j=0}^\infty c_2(\gamma,\alpha)^{j/2}\sqrt{|\mathcal{P}_j|\cdot \sum_{C\in \mathcal{P}_j} p(C)} \\
    &\qquad\qquad \text{ (by Jensen's inequality)}\\
    &\leq \frac{1}{\sqrt{n}} \sum_{j=0}^\infty c_2(\gamma,\alpha)^{j/2}\sqrt{(5\cdot \gamma^{-2/3}B^{1/3}n^{1/3}+1)\cdot 2^j} \\
    &\qquad\qquad\text{(by Lemmas \ref{lem7} and \ref{lem8})}\\
    &\leq \sqrt{\frac{5\cdot B^{1/3}n^{1/3}+1}{\gamma^{2/3}n}} \cdot \frac{1}{1-\sqrt{2c_2(\gamma,\alpha)}} &\\
    &=\frac{\sqrt{5}}{1-\sqrt{2c_2(\gamma,\alpha)}}\cdot\frac{B^{1/6}}{\gamma^{1/3}n^{1/3}}+o(n^{-1/3}) ,&
\end{align*}
and the claim follows since $\alpha$ was picked arbitrarily in $(0,1-1/\gamma)$.
\end{proof}
Theorem \ref{maintheorem} is a direct consequence of propositions $\ref{prop11}$ and $\ref{prop13}$.


\section{Conclusion}
Within the same framework, we can replace the histogram on each set of the partition by a linear estimate with some parameters (slope and intercept at the center point of an interval, for example) only depending upon $N_1,N_2$, and $\lambda(C)$. Such estimates should adapt better to the smoothness of the density and should be studied for the larger class of bounded monotone densities with bounded first derivative.

Estimates of complexity $\kappa > 2$ could lead to nice and simple estimates for convex, concave, log-convex and log-concave densities. For a concave density, for instance, we sketch how one could decide to split a fixed interval $C$. Consider four equal--sized sub-intervals $C_i, 1\leq i\leq 4$, of $C$, and let $N_i$ be the cardinality of $C_i$. If $C$ is not split, we estimate $f$ on $C$ by a linear segment with a slope proportional to $(N_3+N_4)-(N_1+N_2)$.
 If the true density were linear on $C$, then $(N_2+N_3)-(N_1+N_4)$ would be of stochastic order $\sqrt{\sum_i N_i}$, for otherwise it would be positively biased. So, a natural splitting rule would be to split $C$ if 
    \[
        (N_2+N_3)-(N_1+N_4)>\gamma\sqrt{\sum_i N_i}
    \]
    for a fixed design parameter $\gamma$.

Finally, one can easily picture extensions to $[0,1]^d$ for monotone densities (e.g., monotone in each coordinate when all others are fixed). Splitting decisions would then depend upon the $2^d$ cardinalities of all equal quadrants that partition a cell $C$. The splits can be binary (along a preferred dimension) or $2^d$-ary. In the latter case, one would obtain random quadtrees.


\section{Appendix A1: proof of Theorem \ref{nulltheorem}}

Following Devroye and Gy\"orfi (1985, \cite{l1view}) and Devroye (1987, \cite{birkhauser}), it suffices to show that for all
Lebesgue points $x$ with $f(x) > 0$ that $f_n(x) \to f(x)$ in probability.  Here we use the fact that almost all $x$ on $[0,1]$ are Lebesgue points with $f(x) > 0$ (Wheeden and Zygmund, 1977 \cite{wz1977}, 2015 \cite{wz2015}), and recall that $x$ is a Lebesgue point for $f$ if
$$
\lim_{r \downarrow 0} \sup_{y:  x-r \le y \le x+r} \frac{1}{r} \int_y^{y+r} f = f(x).
$$
We fix such an irrational Lebesgue point $x$ in $(0,1)$, and introduce the notation
$C_0,C_1, \ldots$ for the intervals containing $x$ at levels $0, 1, \ldots$ in the binary tree.  Thus, $C_0 = [0,1]$, $C_1$ is 
either $[0,1/2]$ or $[1/2,1]$, and so forth. Let $N_i = N(C_i)$ be
the cardinality of interval $C_i$. Let $K$ be the level at which we find the first leaf on the path to $x$ in the binary tree.  Also, let $U_0, U_1, \ldots$ be the sequence of uniform random variables used for the randomized splitting at each level. In other words, $C_K$ is the first un-split interval, i.e., the sole leaf interval on that path. We first show that $K \to \infty$ and $n/2^K \to \infty$ in probability as $n \to \infty$.

Note that for any large but fixed integer $k$, we have for any integer $m$,
\begin{align*}
\P \{ K \le k \} 
&\le \E \left\{ \sum_{i=0}^k \varphi (N_i) \right\}  \\
&\le (k+1) \E \left\{ \varphi (N_k) \right\}  \\
&\le (k+1) \P \{ N_k \le m \} + (k+1) \varphi (m).
\end{align*}
We can pick $m$ large enough to make the last term as small as desired.
Since $N_k$ is binomial $(n, p_k)$, where $p_k = \int_{C_k} f$,
we have $\P \{ N_k \le m \} = o(1)$.  Therefore, $K \to \infty$ in probability.

Next, for any large but fixed integer $k$, we have for any positive integer $m$
\begin{align*}
\P \{ n/2^K \le k \}
&= \P \{ K \ge \log_2 (n/k) \} \\
&\le \E \left\{ \prod_{i<\log_2 (n/k)} (1 - \varphi (N_i)) \right\}  \\
&\le \E \left\{ \exp \left(- \sum_{i<\log_2 (n/k)} \varphi (N_i) \right) \right\}  \\
&\le \E \left\{ \exp \left(- \sum_{\log_2 (n/(km))\le i<\log_2 (n/k)} \varphi (N_i) \right) \right\}  \\
&\le \E \left\{ \exp \left(- \log_2 (m) \,  \varphi \left( N_{\lfloor{\log_2 (n/(km))}\rfloor}\right) \right) \right\}.
\end{align*}
Now, 
$$
\varphi \left( N_i \right)
\ge \varphi (\ell) 
$$
if $N_i \le \ell$,
where $i= \lfloor{\log_2 (n/(km))}\rfloor$.
Thus,
$$
\P \{ n/2^K \le k \}
\le
\P \{ N_i > \ell \} 
+ e^{- \log_2 m \times  \varphi (\ell) }.
$$
By the Lebesgue density theorem,
$2^i \int_{C_i} f \to f(x)$ as $n$ (and thus $i$) tends to $\infty$. Therefore, there exists a finite constant $c$ such that  $\sup_i \int_{C_i} f \le c/2^i$, and thus,
$$
\E \{ N_i \} = n \int_{C_i} f  
\le \frac{ cn } {2^i }
\le 2ckm.
$$
By Markov's inequality,
\begin{align*}
\P \{ n/2^K \le k \}
&\le\frac{\E \{ N_i \} }{\ell}+ e^{- \log_2 m \times  \varphi (\ell) } \\
&\le \frac {2ckm}{ \ell} + e^{- \log_2 m \times  \varphi (\ell) } .
\end{align*}
We take $\ell = m^2$ and pick $m$ large enough to make the first term small. Since $\log_2 (m^2) \varphi (m^2) \to \infty$, the second term can also be made as small as desired by picking $m$ large enough. We conclude that $n/2^K \to \infty$ in probability.

Let us denote the histogram estimate at $x$ based on the $i$-th level interval $C_i$ by 
$$
g_i (x) = 2^i \frac{N_i}{n}.
$$
For $\epsilon > 0$ and integer $k$, we have
\begin{align*}
&\P \left\{ |f_n (x) - f(x) | > \epsilon \right\} \\
&\quad\le \P \{ K \le k \} + \P \{ K \ge \log_2 (n) -k  \} \\
&\quad \quad
   + \P \left\{ \cup_{i=k}^{\log_2 (n) -  k} 
   \left[ |g_i (x) - f(x) | > \epsilon \right] \right\} .
\end{align*}
By choice of $k$, the first term can be made as small as desired, while the second term is $o(1)$.  The third term is controlled by the union bound,
$$
\sum_{i=k}^{\log_2 (n) -  k} \P \{  |g_i (x) - f(x) | > \epsilon  \}.
$$
Note that
$$
| \E \{ g_i (x) \} - f(x) | 
= \left| 2^i \int_{C_i} f - f(x) \right|
\le \frac{\epsilon}{2}
$$
when $k$ (and thus $i$) is large enough. Using $\V$ to denote the variance, we have
\begin{align*}
 \V \{ g_i (x) \} 
 &= \frac{2^{2i}}{n^2} \V \{N_i\} \\
 &\le \frac{2^{2i} \int_{C_i} f}{n} \\
 &\le \frac{2^{i} (f(x)+\epsilon/2)}{n}.
\end{align*}
So,  by Chebyshev's inequality,
\begin{align*}
\sum_{i=k}^{\log_2 (n) -  k} \P \{  |g_i (x) - f(x) | > \epsilon \}
&\le \sum_{i=k}^{\log_2 (n) -  k} \P \{  |g_i (x) - \E \{ g_i (x) \} | > \epsilon \} \\
&\le \frac{4}{\epsilon^2} \sum_{i=k}^{\log_2 (n) -  k} \V \{  g_i (x) \} \\
&\le \frac{4}{\epsilon^2} \sum_{i=k}^{\log_2 (n) -  k} \frac{2^{i} (f(x)+\epsilon/2)}{n} \\
&\le \frac{8 (f(x)+\epsilon/2) }{2^k \epsilon^2} ,
\end{align*}
and this is as small as desired by picking $k$ large enough.
This concludes the proof of Theorem \ref{nulltheorem}.

\section{Appendix A2: proof of Lemma \ref{lem7}}

    List the unbalanced nodes of $\mathcal{A}_\ell$ in order from right to left, where the leftmost node is that for which the left interval endpoint is the smallest. Denote this list $\{C_i\}_{i=1}^{k}$, where $k=|\mathcal{A}_\ell\setminus \mathcal{B}|$.
    
    By the monotonicity of $f$, we have $p(C_0)\leq p(C_1)\leq\dots\leq p(C_k)$, and we can therefore write $p(C_i)=\sum_{j=0}^iq_j$ for every $i$, where $q_1,\dots,q_k$ are non-negative. Since every $C_i$ is unbalanced, we have $p(C_i')-p(C_i'')>\gamma\sqrt{p(C_i)/n}$ which, combined with the fact that $p(C_i)=p(C_i')+p(C_i'')$, yields $2p(C_i')\geq \gamma\sqrt{p(C_i)/n}+p(C_i)$ and in turn 
    \[
        p(C_{i+1})\geq \gamma\sqrt{\frac{p(C_i)}{n}}+p(C_i)
    \]
    for any $1\leq i\leq k$. We use this fact to prove that for any $1\leq i\leq k$, $q_i\geq ({\gamma^2}/{4n})(i+1)$.
    If $i=0$, we have $q_0=p(C_0)\geq \gamma^2/n\geq \gamma^2/(4n)$. Now assume that the claim regarding $q_i$ holds for some $i$, then
    \begin{align*}
        q_{i+1}&=p(C_{i+1})-p(C_i)\\
        &\geq \gamma\sqrt{\frac{p(C_i)}{n}}\\
        &\geq \frac{\gamma}{\sqrt{n}}\Big(\frac{\gamma^2}{4n}\frac{(i+1)(i+2)}{2}\Big)^{1/2}\\
        &\geq \frac{\gamma^2}{4n}(i+2)
    \end{align*}
    and the claim follows by induction. Therefore,
    \begin{align*}
    p(C_k)=\sum_{i=1}^kq_k\geq \frac{\gamma^2}{4n} \sum_{i=1}^k(i+1) \geq \frac{\gamma^2k^2}{8n},
    \end{align*}
    and the first part of the lemma follows since $p(C_k)\leq B/2^\ell$. The upper bound on $|\mathcal{B}^c|$ follows from the fact that it is no larger than $2^{\ell^*}+\sum_{\ell\geq \ell^*} |\mathcal{A}_\ell\setminus \mathcal{B}|$. Lastly, an identical argument yields the upper bound for $|(\mathcal{B}^{(\alpha)})^c|$.

\section{Appendix A3: proof of Lemma \ref{lem8}}

    We begin by noticing that all but finitely many nodes of $\mathcal{T}_\infty$ are in $\mathcal{B}$. It follows that for any $i\in \mathbb{Z}_{>0}$, $|\mathcal{P}_i|\leq|\mathcal{P}_{i+1}|$ since any node in $\mathcal{P}_i$ is the root of a tree that contains at least one balanced node in $\mathcal{P}_{i+1}$.
    
    Next, we examine how switching a balanced node with its parent affects the various $|\mathcal{P}_i|$'s. Let $C$ be an arbitrary balanced node of $\mathcal{T}_\infty$, $D$ be its parent and $a$ be the number of balanced ancestors of $D$. We may assume that $D$ is unbalanced since switching $D$ and $C$ would leave the tree unaffected. Our operation is depicted in figure \ref{fig:3}. If we let $\mathcal{T}$ be the sub-tree of $\mathcal{T}_\infty$ rooted at $C$'s sibling, then switching $C$ and $D$ applies the map
    \[
        |\mathcal{P}_j|\mapsto \begin{cases}
                                   |\mathcal{P}_j| & \text{if $j\geq a$} \\
                                   |\mathcal{P}_j|-|\mathcal{P}_j(\mathcal{T})|+|\mathcal{P}_{j+1}(\mathcal{T})| & \text{if $j<a$}
  \end{cases}
    \]
    to every $|\mathcal{P}_j|$.
    \begin{figure}
    \centering
    \begin{minipage}{.5\textwidth}
      \centering

\tikzset{every picture/.style={line width=0.75pt}} 

\begin{tikzpicture}[x=0.75pt,y=0.75pt,yscale=-.83,xscale=.83]

\draw    (194.71,1263.05) -- (224.71,1303.05) ;
\draw   (164.71,1303.05) -- (187.74,1347.87) -- (141.69,1347.87) -- cycle ;
\draw    (216.43,1243.98) -- (194.71,1263.05) ;
\draw    (194.71,1263.05) -- (164.71,1303.05) ;
\draw  [fill={rgb, 255:red, 255; green, 255; blue, 255 }  ,fill opacity=1 ] (200.93,1263.05) .. controls (200.93,1259.62) and (198.15,1256.84) .. (194.71,1256.84) .. controls (191.28,1256.84) and (188.5,1259.62) .. (188.5,1263.05) .. controls (188.5,1266.49) and (191.28,1269.27) .. (194.71,1269.27) .. controls (198.15,1269.27) and (200.93,1266.49) .. (200.93,1263.05) -- cycle ;
\draw  [fill={rgb, 255:red, 0; green, 0; blue, 0 }  ,fill opacity=1 ] (218.5,1303.05) .. controls (218.5,1299.62) and (221.28,1296.84) .. (224.71,1296.84) .. controls (228.15,1296.84) and (230.93,1299.62) .. (230.93,1303.05) .. controls (230.93,1306.49) and (228.15,1309.27) .. (224.71,1309.27) .. controls (221.28,1309.27) and (218.5,1306.49) .. (218.5,1303.05) -- cycle ;
\draw   (224.71,1303.05) -- (247.74,1347.87) -- (201.69,1347.87) -- cycle ;
\draw  [dash pattern={on 4.5pt off 4.5pt}]  (233.43,1228.98) -- (216.43,1243.98) ;
\draw    (208.36,1266.97) .. controls (218.71,1272.41) and (225.11,1278.26) .. (229.78,1291.12) ;
\draw [shift={(230.43,1292.98)}, rotate = 251.57] [color={rgb, 255:red, 0; green, 0; blue, 0 }  ][line width=0.75]    (6.56,-1.97) .. controls (4.17,-0.84) and (1.99,-0.18) .. (0,0) .. controls (1.99,0.18) and (4.17,0.84) .. (6.56,1.97)   ;
\draw [shift={(206.43,1265.98)}, rotate = 26.57] [color={rgb, 255:red, 0; green, 0; blue, 0 }  ][line width=0.75]    (6.56,-1.97) .. controls (4.17,-0.84) and (1.99,-0.18) .. (0,0) .. controls (1.99,0.18) and (4.17,0.84) .. (6.56,1.97)   ;
\draw   (252.43,1300.73) -- (276.13,1300.73) -- (276.13,1297.98) -- (284.43,1303.48) -- (276.13,1308.98) -- (276.13,1306.23) -- (252.43,1306.23) -- cycle ;
\draw    (338.71,1264.05) -- (368.71,1304.05) ;
\draw   (308.71,1304.05) -- (331.74,1348.87) -- (285.69,1348.87) -- cycle ;
\draw    (338.71,1264.05) -- (308.71,1304.05) ;
\draw  [fill={rgb, 255:red, 0; green, 0; blue, 0 }  ,fill opacity=1 ] (344.93,1264.05) .. controls (344.93,1260.62) and (342.15,1257.84) .. (338.71,1257.84) .. controls (335.28,1257.84) and (332.5,1260.62) .. (332.5,1264.05) .. controls (332.5,1267.49) and (335.28,1270.27) .. (338.71,1270.27) .. controls (342.15,1270.27) and (344.93,1267.49) .. (344.93,1264.05) -- cycle ;
\draw   (368.71,1304.05) -- (391.74,1348.87) -- (345.69,1348.87) -- cycle ;
\draw    (352.36,1267.97) .. controls (362.71,1273.41) and (369.11,1279.26) .. (373.78,1292.12) ;
\draw [shift={(374.43,1293.98)}, rotate = 251.57] [color={rgb, 255:red, 0; green, 0; blue, 0 }  ][line width=0.75]    (6.56,-1.97) .. controls (4.17,-0.84) and (1.99,-0.18) .. (0,0) .. controls (1.99,0.18) and (4.17,0.84) .. (6.56,1.97)   ;
\draw [shift={(350.43,1266.98)}, rotate = 26.57] [color={rgb, 255:red, 0; green, 0; blue, 0 }  ][line width=0.75]    (6.56,-1.97) .. controls (4.17,-0.84) and (1.99,-0.18) .. (0,0) .. controls (1.99,0.18) and (4.17,0.84) .. (6.56,1.97)   ;
\draw  [fill={rgb, 255:red, 255; green, 255; blue, 255 }  ,fill opacity=1 ] (170.93,1303.05) .. controls (170.93,1299.62) and (168.15,1296.84) .. (164.71,1296.84) .. controls (161.28,1296.84) and (158.5,1299.62) .. (158.5,1303.05) .. controls (158.5,1306.49) and (161.28,1309.27) .. (164.71,1309.27) .. controls (168.15,1309.27) and (170.93,1306.49) .. (170.93,1303.05) -- cycle ;
\draw  [fill={rgb, 255:red, 255; green, 255; blue, 255 }  ,fill opacity=1 ] (314.93,1304.05) .. controls (314.93,1300.62) and (312.15,1297.84) .. (308.71,1297.84) .. controls (305.28,1297.84) and (302.5,1300.62) .. (302.5,1304.05) .. controls (302.5,1307.49) and (305.28,1310.27) .. (308.71,1310.27) .. controls (312.15,1310.27) and (314.93,1307.49) .. (314.93,1304.05) -- cycle ;
\draw  [fill={rgb, 255:red, 255; green, 255; blue, 255 }  ,fill opacity=1 ] (362.5,1303.8) .. controls (362.5,1300.51) and (365.17,1297.84) .. (368.46,1297.84) .. controls (371.76,1297.84) and (374.43,1300.51) .. (374.43,1303.8) .. controls (374.43,1307.1) and (371.76,1309.77) .. (368.46,1309.77) .. controls (365.17,1309.77) and (362.5,1307.1) .. (362.5,1303.8) -- cycle ;
\draw    (360.43,1244.98) -- (338.71,1264.05) ;
\draw  [dash pattern={on 4.5pt off 4.5pt}]  (377.43,1229.98) -- (360.43,1244.98) ;

\draw (170,1253.7) node [anchor=north west][inner sep=0.75pt]  [font=\footnotesize]  {$D$};
\draw (201,1297.7) node [anchor=north west][inner sep=0.75pt]  [font=\footnotesize]  {$C$};
\draw (156,1320.7) node [anchor=north west][inner sep=0.75pt]    {$\mathcal{T}$};
\draw (300,1321.7) node [anchor=north west][inner sep=0.75pt]    {$\mathcal{T}$};

\end{tikzpicture}    
   \captionsetup{justification=centering}
\captionof{figure}{Switching a balanced \newline node with its parent.}
      \label{fig:3}
    \end{minipage}%
    \begin{minipage}{.5\textwidth}
      \centering
     
\tikzset{every picture/.style={line width=0.75pt}} 

\begin{tikzpicture}[x=0.75pt,y=0.75pt,yscale=-1,xscale=1]

\draw  [fill={rgb, 255:red, 255; green, 255; blue, 255 }  ,fill opacity=1 ] (146.58,1356.05) -- (196.74,1453.68) -- (96.43,1453.68) -- cycle ;
\draw    (118,1412.27) .. controls (125.43,1418.55) and (134.58,1402.87) .. (144.58,1406.87) .. controls (154.58,1410.87) and (162.43,1439.55) .. (188.43,1435.55) ;
\draw  [fill={rgb, 255:red, 255; green, 255; blue, 255 }  ,fill opacity=1 ] (138.73,1356.05) .. controls (138.73,1351.71) and (142.24,1348.19) .. (146.58,1348.19) .. controls (150.93,1348.19) and (154.44,1351.71) .. (154.44,1356.05) .. controls (154.44,1360.39) and (150.93,1363.91) .. (146.58,1363.91) .. controls (142.24,1363.91) and (138.73,1360.39) .. (138.73,1356.05) -- cycle ;
\draw  [fill={rgb, 255:red, 255; green, 255; blue, 255 }  ,fill opacity=1 ] (26.43,1367.55) -- (43.43,1367.55) -- (43.43,1377.27) -- (26.43,1377.27) -- cycle ;
\draw   (26.43,1387.55) -- (43.43,1387.55) -- (43.43,1397.27) -- (26.43,1397.27) -- cycle ;
\draw    (140.43,1406.84) -- (99.43,1447.84) ;
\draw    (147.4,1407.87) -- (102.57,1452.7) ;
\draw    (151.04,1412.41) -- (110.75,1452.7) ;
\draw    (154.43,1416.02) -- (116.75,1453.7) ;
\draw    (157.43,1421.02) -- (124.43,1454.02) ;
\draw    (127.18,1412.27) -- (108.43,1431.02) ;
\draw    (161.04,1425.41) -- (132.43,1454.02) ;
\draw    (165.75,1429.7) -- (141,1454.45) ;
\draw    (170.75,1432.7) -- (149,1454.45) ;
\draw    (183.75,1435.7) -- (165,1454.45) ;
\draw    (188.75,1438.7) -- (173.75,1453.7) ;
\draw    (190.75,1443.7) -- (180.75,1453.7) ;
\draw    (193.75,1448.7) -- (188.75,1453.7) ;
\draw    (119,1413.27) -- (115,1417.27) ;
\draw    (176.75,1434.7) -- (156,1455.45) ;
\draw  [color={rgb, 255:red, 255; green, 255; blue, 255 }  ,draw opacity=1 ][fill={rgb, 255:red, 255; green, 255; blue, 255 }  ,fill opacity=1 ] (89.88,1448.52) -- (210,1448.52) -- (210,1458.52) -- (89.88,1458.52) -- cycle ;
\draw    (43.43,1367.55) -- (33.47,1377.52) ;
\draw    (37.39,1367.31) -- (27.43,1377.27) ;
\draw    (31,1367.98) -- (26.43,1372.55) ;
\draw    (43.43,1373.27) -- (39.39,1377.31) ;

\draw (48,1366.27) node [anchor=north west][inner sep=0.75pt]  [font=\footnotesize]  {$Balanced$};
\draw (47.43,1386.27) node [anchor=north west][inner sep=0.75pt]  [font=\footnotesize]  {$Unbalanced$};

\end{tikzpicture}
\captionsetup{justification=centering}
      \captionof{figure}{Tree for which $|\mathcal{P}_j|$ \newline
      is maximized (for all $j$).}
      \label{fig:4}
    \end{minipage}
    \end{figure}

    Since $|\mathcal{P}_j(\mathcal{T})|\leq|\mathcal{P}_{j+1}(\mathcal{T})|$ for any $j$, this map's output is always greater than or equal to $|\mathcal{P}_j|$. In other words, the node configuration that maximizes $|\mathcal{P}_j|$ for every $j$ is such that all the unbalanced nodes are pushed to the top. This forms a tree of unbalanced nodes whose leaves are the roots of full, infinite binary trees where all nodes are balanced.
    
    In this configuration, it is clear that $|\mathcal{P}_0|\leq (|\mathcal{B}^c|+1|)$ and that $|\mathcal{P}_{j+1}|\leq 2|\mathcal{P}_j|$, from which the lemma follows.

\section{Appendix A4: proof of Lemma \ref{lem9}}

    It suffices to show that for all pairs of disjoint intervals $C_1, C_2 \subseteq C$ such that $C_1\cup C_2=C$
    \[
        \int_{C_1}\big(f-f(C_1)\big)_+ + \int_{C_2}\big(f-f(C_2)\big)_+\leq \int_{C}\big(f-f(C)\big)_+.
    \]
    The lemma then follows by induction. Without loss of generality, assume that $C=[0,1]$.
    
    Let $C_1$, $C_2$ be an arbitrary such pair and note that $f(C_2)\leq f(C) \leq f(C_1)$ since $f$ is decreasing. Let $x_1 := \sup\{x:f(x)\geq f(C_1)\}$, $x_2:=\sup\{x:f(x)\geq f(C_2)\}$, and let $A:= \int_{0}^{x_1}(f-f(C_1))$ and $B:=\int_{x_2}^1(f-f(C_2))$ as depicted in figure \ref{fig:3.5}. Then it is obvious that
    \[
        A+B = \int_{C_1} \big(f-f(C_1)\big)_+ + \int_{C_2} \big(f-f(C_2)\big)_+ \leq \int_C\big(f-f(C)\big)_+.
    \]

\section{Appendix A5: proof of Lemma \ref{lem10}}

    Let $C\in T_n$ be arbitrary. Recall that $N(C)$ is the number of data points lying in $C$ (similarly, $N(C_1)$ and $N(C_2)$ are the number of points lying in the left/right halves of $C$ respectively). Notice that $N(C) \stackrel{\mathcal{L}}{=} \Bin (n, p(C))$ (where $\stackrel{ \mathcal{L}}{=}$ denotes equivalence in law and $\Bin(n,p)$ a binomial $n$, $p$). For $C$ to be a leaf, we must have decided not to split its node. Therefore
    \begin{align}
        \P\{C\in L\} &\leq \P\Big\{N(C_1)-N(C_2)<\gamma\sqrt{N(C)}\Big\}\nonumber\\
                    &\leq \sup_{m\geq np(C)/2}\P\bigg\{{N(C_1)-N(C_2)}<{\gamma\sqrt{N(C)}} \,\, \Big| \,\, N(C)=m\bigg\}\nonumber\\  &\quad\quad+\P\big\{|N(C)-np(C)|>np(C)/2\big\}. \label{sup}
    \end{align}
    A simple application of Chebyshev's inequality gives
    \[
        \P\big\{|N(C)-np(C)|>np(C)/2\big\}\leq \frac{4(1-p(C))}{np(C)}\leq \frac{4}{np(C)}.
    \]
    Next, given $N(C)=m$, $N(C_1)\stackrel{\mathcal{L}}{=}\Bin(m, p(C_1)/p(C))$. Note that 
    \begin{align*}
        &p(C_1)+p(C_2)=p(C)\\
        &p(C_1)-p(C_2)=\gamma\sqrt{\frac{2p(C)}{n}}+\xi(C)
    \end{align*}
    so that
    \[
        \frac{p(C_1)}{p(C)}=\frac{1}{2}+\frac{1}{2}\gamma\sqrt{\frac{2}{np(C)}}+\frac{1}{2}\frac{\xi(C)}{p(C)}.
    \]
    Using these observations and the fact that $N(C_1)-N(C_2)=2N(C_1)-N(C)$, we find that
    \begin{align*}
        \P&\bigg\{{N(C_1)-N(C_2)}<{\gamma\sqrt{N(C)}} \,\, \Big| \,\, N(C)=m\bigg\}\\
          &=\P\Bigg\{\Bin\bigg(m,\frac{p(C_1)}{p(C)}\bigg)-m\frac{p(C_1)}{p(C)}<\gamma\frac{\sqrt{m}}{2}-\gamma\frac{m}{2}\sqrt{\frac{2}{np(C)}}-\frac{m}{2}\frac{\xi(C)}{p(C)}\Bigg\}\\
          &\leq \P\Bigg\{\Bin\bigg(m,\frac{p(C_1)}{p(C)}\bigg)-m\frac{p(C_1)}{p(C)}<-\frac{m}{2}\frac{\xi(C)}{p(C)}\Bigg\}
    \end{align*}
    if $m\geq np(C)/2$, which is the case in the supremum taken in (\ref{sup}). Using the Chebyshev-Cantelli inequality (see Lugosi, Massart and Boucheron, 2013 \cite{lugosi}) and the fact that the variance of a binomial $n$, $p$ is at most $n/4$, we have
    \begin{align*}
        &\P\Bigg\{\Bin\bigg(m,\frac{p(C_1)}{p(C)}\bigg)-m\frac{p(C_1)}{p(C)}<-\frac{m}{2}\frac{\xi(C)}{p(C)}\Bigg\} \\
        &\qquad\leq \frac{m/4}{m/4+\frac{m^2\xi(C)^2}{4(p(C))^2}} \\
        &\qquad\leq \frac{2p(C)}{2p(C)+n\xi(C)^2} \quad\bigg(\text{since $m\geq \frac{np(C)}{2}$}\bigg)
    \end{align*}
    and the lemma follows.

\section{Appendix A6: proof of Lemma \ref{lem12}}

    Using the Cauchy-Schwarz inequality, we write
    \begin{align*}
        \E\bigg\{\ind{C\in L}\Big(p(C)-\frac{N(C)}{n}\Big)\bigg\}\leq \sqrt{\P\{C\in L, N(C)\leq np(C)\}}\cdot\sqrt{\frac{p(C)}{n}}.
    \end{align*}
    We claim that 
    \[
        \P\{C\in L, N(C)\leq np(C)\} \leq c_2(\gamma,\alpha)^j,
    \]
    where $c_2(\gamma,\alpha)$ is defined in the lemma's statement. This bound does not depend on the level $\ell$ at which the node is located. To prove this, it suffices to show that for any balanced node $C\in \mathcal{B}^{(\alpha)}$, if $N(C)\leq np(C)$,
    \[
        \P\{N(C_1)-N(C_2)>\gamma \sqrt{N(C)}\,\,|\,\, N(C)\} \leq c_2(\gamma,\alpha),
    \]
    since consecutive splits are independent given $N(C)$. For simplicity, temporarily denote $N(C)$, $N(C_1)$ and $N(C_2)$ by $N, N_1$ and $N_2$ respectively, and similarly for $p, p_1, p_2$. Then observe that
    \begin{align*}
        \P\Big\{N_1&-N_2>\gamma \sqrt{N}\,\,\Big|\,\, N\Big\} \\
        &\leq \P\bigg\{N_1-N_2-N\Big(\frac{p_1-p_2}{p}\Big)>\gamma \sqrt{N}-N\Big(\frac{p_1-p_2}{p}\Big)\,\,\bigg|\,\, N\bigg\}\\
        &\leq \ind{N(p_1-p_2)/p\geq \gamma\alpha\sqrt{N}}+\P\bigg\{N_1-N_2-N\Big(\frac{p_1-p_2}{p}\Big)>\gamma(1-\alpha) \sqrt{N}\,\,\bigg|\,\, N\bigg\}\\
        &\leq \ind{\sqrt{N}>(\gamma \alpha)(p/(p_1-p_2))}+\P\left\{\frac{N_1-N_2-N\cdot\frac{p_1-p_2}{p}}{\sqrt{4N\cdot\frac{p_1}{p}\frac{p_2}{p}}}>\gamma(1-\alpha)\frac{p/2}{\sqrt{p_1p_2}}\,\,\bigg|\,\, N\right\}. 
    \end{align*}
    By definition of $\mathcal{B}^{(\alpha)}$, we have $p_1-p_2<(\gamma\alpha)\sqrt{p/n}$. If the indicator in the expression above were one, this would imply that $\sqrt{N}>\sqrt{np}$ which cannot be true since $N\leq np$, hence the indicator is equal to 0. As for the probability term, notice that
    \[
        \sqrt{4N\cdot\frac{p_1}{p}\frac{p_2}{p}}
    \]
    is the conditional variance of 
    \[
        N_1-N_2-N\cdot\frac{p_1-p_2}{p},
    \]
    which is a random variable with mean 0 and unit variance. Noting that $p/2\geq \sqrt{p_1p_2}$ and that $\gamma(1-\alpha)>1$ by assumption, we can apply the Chebyshev-Cantelli inequality to get
    \[
        \P\left\{\frac{N_1-N_2-N\cdot\frac{p_1-p_2}{p}}{\sqrt{4N\cdot\frac{p_1}{p}\frac{p_2}{p}}}>\gamma(1-\alpha)\frac{p/2}{\sqrt{p_1p_2}}\,\,\bigg|\,\, N\right\}\leq\frac{1}{1+(\gamma-\alpha\gamma)^2}  \stackrel{\text{def}}{=} c_2(\gamma,\alpha)
    \]
    and $c_1(\gamma,\alpha)<1/2$.
\section{Appendix A7: proof of corollary \ref{cor2}}

The decision to split an interval containing $k$ points can be made in order $\log_2(k)$ time, which is the time taken to determine which points lie on the left and right halves of the interval, respectively, via binary search.
It therefore suffices to show that the expected number of leaves of the tree generated by the algorithm is of the order of $n^{1/3}$. Letting $L$ denote the tree's leaf set and $\alpha=(\gamma-1)/2\gamma$, we have
\begin{align*}
    \E\{|L|\} 
    &\leq \E\big\{\big|\big(\mathcal{B}^{(\alpha)}\big)^c\big|\big\} + \E\big\{\big|L\cap \mathcal{B}^{(\alpha)}\big|\big\} \\
    &= \E\big\{\big|\big(\mathcal{B}^{(\alpha)}\big)^c\big|\big\}   + \sum_{j=0}^\infty \sum_{C\in \mathcal{B}^{(\alpha)} \cap \mathcal{P}_j} \E\big\{\ind{C\in L}\big\}.
\end{align*}
The first term is $O(n^{1/3})$ by Lemma \ref{lem7} and the second by the proofs of Proposition \ref{prop13} and Lemma \ref{lem12}.

\section{Acknowledgments}
The authors would like to thank both referees.

\nocite{*}
\bibliographystyle{abbrv}
\bibliography{biblio}

\appendix

\end{document}